\documentclass[10pt,leqno]{article}
\usepackage[paperwidth=6.75in,paperheight=10in,textwidth=5.125in,
  textheight=598pt,top=60.6pt,centering,headheight=10pt,headsep=14.454pt,
  footskip=14.454pt]{geometry}
\usepackage{amsmath,amsfonts,amssymb,mathrsfs}
\usepackage{mathtools,bm,graphicx}
\usepackage[amsmath,thmmarks,hyperref]{ntheorem}

\makeatletter
\renewcommand{\normalsize}{%
   \@setfontsize\normalsize\@xpt\@xiipt
   \abovedisplayskip 10\p@ \@plus2\p@ \@minus5\p@
   \abovedisplayshortskip \z@ \@plus3\p@ 
   \belowdisplayshortskip 6\p@ \@plus3\p@ \@minus3\p@ 
   \belowdisplayskip \abovedisplayskip
   \let\@listi\@listI}
\normalsize
\renewcommand{\small}{%
   \@setfontsize\small\@ixpt{11}%
   \abovedisplayskip 8.5\p@ \@plus3\p@ \@minus4\p@
   \abovedisplayshortskip \z@ \@plus2\p@
   \belowdisplayshortskip 4\p@ \@plus2\p@ \@minus2\p@
   \def\@listi{\leftmargin\leftmargini
               \topsep \z@ 
               \parsep \z@ 
               \itemsep \parsep}%
   \belowdisplayskip \abovedisplayskip
}
\renewcommand{\footnotesize}{%
   \@setfontsize\footnotesize\@viiipt{9.5}%
   \abovedisplayskip 6\p@ \@plus2\p@ \@minus4\p@
   \abovedisplayshortskip \z@ \@plus\p@ 
   \belowdisplayshortskip 3\p@ \@plus\p@ \@minus2\p@ 
   \def\@listi{\leftmargin\leftmargini
               \topsep \z@
               \parsep \z@
               \itemsep \parsep}%
   \belowdisplayskip \abovedisplayskip
}
\renewcommand{\scriptsize}{\@setfontsize\scriptsize\@viipt\@viiipt}
\renewcommand{\tiny}{\@setfontsize\tiny\@vpt\@vipt}
\renewcommand{\large}{\@setfontsize\large\@xiipt{14}}
\renewcommand{\Large}{\@setfontsize\Large\@xivpt{18}}
\renewcommand{\LARGE}{\@setfontsize\LARGE\@xviipt{22}}
\renewcommand{\huge}{\@setfontsize\huge\@xxpt{25}}
\renewcommand{\Huge}{\@setfontsize\Huge\@xxvpt{30}}

\skip\footins 13pt plus 4pt minus 2pt
\def\@listI{\leftmargin\leftmargini\parsep\z@\topsep\z@\itemsep\parsep}
\let\@listi\@listI
\@listi
\def\@listii{\leftmargin\leftmarginii\labelwidth\leftmarginii
  \advance\labelwidth-\labelsep\topsep\z@\parsep\z@\itemsep\parsep}
\def\@listiii{\leftmargin\leftmarginiii\labelwidth\leftmarginiii
  \advance\labelwidth-\labelsep\topsep\z@\parsep\z@\partopsep\z@\itemsep\topsep}
\def\@seccntformat#1{\csname the#1\endcsname.\hskip .5em}
\let\preprint@sect\@sect
\def\@sect#1#2#3#4#5#6[#7]#8{%
  \preprint@sect{#1}{#2}{#3}{#4}{#5}{#6}[#7]{#8.}}
\renewcommand\section{\@startsection{section}{1}{.25in}%
  {1.3ex plus .5ex minus .2ex}{-.5em plus -.1em}{\normalfont\normalsize\bfseries}}
\renewcommand\subsection{\@startsection{subsection}{2}{.25in}%
  {1.3ex plus .5ex minus .2ex}{-.5em plus -.1em}{\normalfont\normalsize\bfseries}}
\renewcommand\subsubsection{\@startsection{subsubsection}{3}{.25in}%
  {1.3ex plus .5ex minus .2ex}{-.5em plus -.1em}{\normalfont\normalsize\bfseries}}

\def\@maketitle{%
  \null
  \begin{center}%
    {\bfseries\@ucnothanks{\@title}\par}%
    \vskip .075in%
    {\footnotesize
      \spaceskip=1.5\fontdimen2\the\font
      \def\and{\ \scriptsize AND }
      \@tempswafalse\expandafter\@doauth\@author\and\relax}%
  \end{center}%
  \par
  \vskip .11in}
\def\@ucnothanks#1{\expandafter\@ucnt#1\thanks\relax}
\def\@ucnt#1\thanks{\uppercase\expandafter{#1}\futurelet\@tempa\@ucnta}
\def\@ucnta{\ifx\@tempa\relax\else\@ucntb\fi}
\def\@ucntb#1\thanks\relax{\thanks{#1}}
\def\@doauth#1\and{\let\@nextdoa=\@doauth
 \def\@tempc{\@ucnothanks{#1}}
 \futurelet\@tempa\@doautha}
\def\@doautha{\ifx\@tempa\relax
  \if@tempswa {\and}\fi
  \let\@nextdoa=\relax
 \else
  \if@tempswa \gdef\and{,\ \scriptsize AND }, \fi
 \fi
 \@tempc\unskip\@tempswatrue\@nextdoa}
\let\and=\relax

\renewcommand\maketitle{\par\begingroup
  \renewcommand{\thefootnote}{\fnsymbol{footnote}}%
  \def\@makefnmark{\hbox to\z@{$\m@th^{\@thefnmark}$\hss}}%
  \long\def\@makefntext##1{\parindent 1em\noindent
    \hbox to1.8em{\hss$\m@th^{\@thefnmark}$}##1}%
  \global\@topnum\z@\@maketitle\thispagestyle{plain}\@thanks
  \endgroup\setcounter{footnote}{0}%
  \let\maketitle\relax\let\@maketitle\relax
  \gdef\@thanks{}\gdef\@author{}\gdef\@title{}\let\thanks\relax}
\renewenvironment{abstract}{\vspace{.05in}\footnotesize\parindent .2in
  {\upshape\bfseries Abstract. }\ignorespaces}{\par\vspace{.1in}}
\renewcommand\footnoterule{\kern-3pt\hrule width .5in\kern2.6pt}

\long\def\@makecaption#1#2{%
  \footnotesize\setlength{\parindent}{1.5pc}\vskip\abovecaptionskip
  \setbox\@tempboxa\hbox{{\normalfont\scshape #1}. {\normalfont\itshape #2}}%
  \ifdim\wd\@tempboxa>\hsize
    {\normalfont\scshape #1}. {\normalfont\itshape #2}\par
  \else\global\@minipagefalse\hb@xt@\hsize{\hfil\box\@tempboxa\hfil}\fi}

\makeatother

\usepackage[hidelinks]{hyperref}
\usepackage[capitalize,nameinlink]{cleveref}
\newcommand{\E}{\mathbb E}
\newcommand{\Pp}{\mathbb P}
\newcommand{\R}{\mathbb R}

\newcommand{\eps}{\varepsilon}
\newcommand{\cF}{\mathcal F}

\newcommand{\1}{\mathbf 1}
\newcommand{\normF}[1]{\lVert #1\rVert_{\mathrm F}}

\makeatletter
\renewtheoremstyle{plain}
 {\item[\hskip\labelsep\hskip\parindent \theorem@headerfont ##1\ \textup{##2\theorem@separator}]}%
 {\item[\hskip\labelsep\hskip\parindent \theorem@headerfont ##1\ \textup{##2}]\textup{(##3)\theorem@separator}\ }

\renewtheoremstyle{nonumberplain}%
 {\item[\theorem@headerfont\hskip\labelsep\hskip\parindent ##1\theorem@separator]}%
{\item[\theorem@headerfont\hskip\labelsep\hskip\parindent ##3\theorem@separator]}

\newcommand{\proofbox}{\vbox{\hrule height0.6pt\hbox{\vrule height1.3ex width0.6pt\hskip0.8ex\vrule width0.6pt}\hrule height0.6pt}}

\theorempreskip{1ex plus .25ex minus .1ex}
\theorempostskip{1ex plus .25ex minus .1ex}

\theoremstyle{nonumberplain}
\theoremheaderfont{\normalfont\itshape}
\theorembodyfont{\normalfont}
\theoremseparator{.}
\theoremsymbol{\proofbox}
\newtheorem{proof}{Proof}

\theoremstyle{plain}
\theoremheaderfont{\normalfont\sc}
\theorembodyfont{\normalfont\itshape}
\theoremseparator{.}
\theoremsymbol{}
\newtheorem{theorem}{Theorem}

\newcommand{\newsiamthm}[2]{
  \theoremstyle{plain}
  \theoremheaderfont{\normalfont\sc}
  \theorembodyfont{\normalfont\itshape}
  \theoremseparator{.}
  \theoremsymbol{}
  \newtheorem{#1}[theorem]{#2}
}

\newsiamthm{lemma}{Lemma}
\newsiamthm{corollary}{Corollary}
\newsiamthm{proposition}{Proposition}
\newsiamthm{definition}{Definition}

\newcommand{\newsiamremark}[2]{
  \theoremstyle{plain}
  \theoremheaderfont{\normalfont\itshape}
  \theorembodyfont{\normalfont}
  \theoremseparator{.}
  \theoremsymbol{}
  \newtheorem{#1}[theorem]{#2}
}

\newsiamremark{assumption}{Assumption}
\newsiamremark{remark}{Remark}
\newsiamremark{example}{Example}
\makeatother
\numberwithin{theorem}{section}
\numberwithin{equation}{section}
\crefname{assumption}{Assumption}{Assumptions}
\crefname{remark}{Remark}{Remarks}
\crefname{example}{Example}{Examples}
\AddToHook{env/lemma/begin}{\crefalias{theorem}{lemma}}
\AddToHook{env/corollary/begin}{\crefalias{theorem}{corollary}}
\AddToHook{env/proposition/begin}{\crefalias{theorem}{proposition}}
\AddToHook{env/definition/begin}{\crefalias{theorem}{definition}}
\newcommand{\email}[1]{\href{mailto:#1}{#1}}

\title{Exponential Stability of Slow--Fast Stochastic Systems with Delay}

\author{Yudong Wang\thanks{Research Institute of Intelligent Complex Systems,
Fudan University, Shanghai 200433, China.}
\and Kele Sun\thanks{Research Institute of Intelligent Complex Systems,
Fudan University, Shanghai 200433, China.}
\and Shijie Zhou\thanks{Corresponding author. Research Institute of Intelligent
Complex Systems, Fudan University, Shanghai 200433, China
(\email{sjzhou14@fudan.edu.cn}).}
\and Xuerong Mao\thanks{Department of Mathematics and Statistics,
University of Strathclyde, Glasgow G1 1XH, UK.}}

\date{}
\hypersetup{
  pdftitle={Exponential Stability of Slow--Fast Stochastic Systems with Delay},
  pdfauthor={Yudong Wang, Kele Sun, Shijie Zhou, Xuerong Mao}
}

\begin{document}
\maketitle

\begin{abstract}
We develop a {novel cutoff Lyapunov method} to establish low-order moment stability $(p\in(0,2))$ for mutually coupled slow–fast stochastic systems with delay. The interplay of mutual coupling, delayed diffusion, and {singular derivatives of low-order Lyapunov functions} creates difficulties that cannot be resolved by directly combining existing averaging and delay-stability approaches. Our method estimates the averaging and delay errors through an adaptive cutoff, yielding Lyapunov estimates that can be iterated over an infinite time horizon {without requiring strong trajectory comparisons}. Under an averaged Lyapunov condition and for sufficiently small values of the delay and the time-scale parameter, we establish global practical $p$th-moment exponential stability, followed by two distinct mechanisms for moment and almost-sure exponential stability: homogeneous fast dissipation yields joint slow–fast stability, while vanishing fast-state sensitivity ensures slow stability despite persistent fast fluctuations. We further extend the approach to locally Lipschitz systems with superlinear coefficients, including stochastic delay equations without fast components, establishing exponential stability with arbitrarily high probability for bounded sets of initial value.

  \end{abstract}


\section{Introduction}

Slow--fast stochastic systems arise in many applications,
including chemical kinetics, neuroscience, and climate
dynamics~\cite{GaltierWainrib2012,MajdaTimofeyevVandenEijnden2001,WuTianRawlingsYin2016}.
Stochastic averaging simplifies such systems by replacing
the fast dynamics with their long-term statistical effects
on the slow component~\cite{Khasminskii1968,NguyenYinZhu2025,PavliotisStuart2008,YinZhang2013}.
Averaging principles have been established in both the
strong~\cite{BaoYinYuan2017,KhasminskiiYin2005,LiuRocknerSunXie2020,RocknerSunXie2019}
and weak senses~\cite{Kushner1984,MaoShao2024,RocknerXie2021}.
However, when the slow diffusion depends on the fast
variable, strong convergence to the averaged SDE generally
fails, making weak averaging more appropriate~\cite{Kifer2003}.

Averaging principles alone are insufficient to establish exponential stability, even with uniform-in-time approximations \cite{Crisan2024,XieZhang2026}, since the approximation errors need not decay over time.
Early studies investigated stability under wide-band noise
and rapidly varying environments
\cite{BlankenshipPapanicolaou1978,Korolyuk1998,KorolyukChabanyuk2002}.
Yin and Zhang~\cite{YinZhang2003} studied stability for
two-time-scale Markov chains using perturbed Lyapunov
functions, while Liu and Krsti\'c~\cite{LiuKrstic2010}
developed infinite-time averaging for locally Lipschitz
systems. Khasminskii~\cite{Khasminskii2012} investigated
stability under rapid Markov switching, and more recently,
Liu et al.~\cite{Liu2025} studied stability for
slow--fast SDEs with fast-dependent diffusion.
Nevertheless, important questions remain concerning
noise-induced stabilisation in slow--fast stochastic systems.

Noise-induced stabilisation is well known
\cite{ArnoldCrauelWihstutz1983,CaraballoGarridoAtienzaReal2003,Hasminskii1980,Scheutzow1985,Scheutzow1993}.
For example, the scalar equation
${\rm d}Z_t=aZ_t\,{\rm d}t+\sigma Z_t\,{\rm d}W_t$
with $\sigma^2>2a>0$ satisfies
$
\E|Z_t|^p=|Z_0|^p
\exp\left\{p\left(a-\frac{1-p}{2}\sigma^2\right)t\right\}.
$
Exponential $p$th-moment
stability holds iff
$0<p<1-2a/\sigma^2<1$.
This illustrates why low-order moments are
essential in noise-induced stabilisation
and motivates stochastic feedback control
~\cite{ApplebyMaoRodkina2008,DengLuoMao2012,Mao1994,MaoYinYuan2007,WuHu2009,ZhaoDengFu2020}.

Small delays are common in feedback systems
~\cite{MeinsmaFuIwasaki1999,VyhlidalMichielsZitek2009},
yet even arbitrarily small delays can destroy stability
~\cite{Datko1988,LogemannRebarberWeiss1996},
including fourth-moment exponential stability in stochastic
systems~\cite{HuLiuDengMao2020}.
This motivates the study of stability robustness under
small delays~\cite{ApplebyMao2005,DongMao2022,LiMao2020,LuHuMao2019,Mao1996,MaoLamHuang2008,MaoShah1997}.
Zong et al.~\cite{ZongLiYinZhang2021} developed a systematic
delay-tolerance framework, mainly for $p\ge2$.
For low-order moments, Guo et al.~\cite{GuoMaoYue2016},
Hu et al.~\cite{HuLiuDengMao2020}, and
Zhao and Deng~\cite{ZhaoDeng2024} established stability
transfer results by comparing delayed and nondelayed
systems in $L^2$.


In this paper, for $0<\eps\le1$, we consider the following stochastic slow--fast system
\begin{small}\begin{equation}\label{eq-1.1}
\begin{cases}
{\rm d}X_t^\eps
=b(X_t^\eps,Y_t^\eps)\,{\rm d}t
+\sigma(X_{t-\tau}^\eps,Y_t^\eps)\,{\rm d}B_t^1,
& t\ge0,\\[1mm]
{\rm d}Y_t^\eps
=\dfrac1\eps f(X_t^\eps,Y_t^\eps)\,{\rm d}t
+\dfrac1{\sqrt\eps}g(X_t^\eps,Y_t^\eps)\,{\rm d}B_t^2,
& t\ge0.
\end{cases}
\end{equation}\end{small}
Here $\tau>0$ is the delay, the initial history of the slow component is
$X_t^\eps=\phi(t)$ for $-\tau\le t\le0$, where $\phi\in C([-\tau,0],\R^d)$, the initial value of the fast component is $Y_0^\eps=y$, and $B^1$ and $B^2$ are independent standard Brownian motions. The coefficient maps are
$
b:\R^d\times\R^m\to\R^d$,
$\sigma:\R^d\times\R^m\to\R^{d\times d_1}
$,
$
f:\R^d\times\R^m\to\R^m$,$
g:\R^d\times\R^m\to\R^{m\times d_2}.
$
Define the operator for the slow component
\begin{small}
\begin{equation}\label{eq-3.34}
\mathscr L_y\varphi(x)
:=\langle b(x,y),\nabla\varphi(x)\rangle
+\frac12\operatorname{tr}\bigl(a(x,y)D^2\varphi(x)\bigr).
\end{equation}
\end{small} Suppose that, for some $p\in(0,2)$, the averaged
system~\eqref{eq-3.22} admits a Lyapunov function
$V(x)\asymp |x|^p$. Our aim is to determine conditions under
which its stability transfers to the coupled
system~\eqref{eq-1.1} for sufficiently small $\eps$ and $\tau$. However, doing so is not straightforward.


A key difficulty is to control the averaging error in the mutually
coupled system using low-order Lyapunov functions. For $0<p<2$,
functions such as $V(x)\asymp|x|^p$ have singular derivatives near the
origin. The drift and diffusion fluctuations relative to the averaged
coefficients are weighted by these derivatives, which obstructs a
direct application of standard perturbed-Lyapunov
arguments \cite{Liu2025,LiuKrstic2010}, even in the absence of delay.

The delay in the slow diffusion introduces an additional error,
since the averaged Lyapunov condition involves only the current state.
Existing low-order moment delay-stability
results~\cite{GuoMaoYue2016,HuLiuDengMao2020,ZhaoDeng2024}
rely on $L^2$ comparisons with a stable nondelayed system.
This approach is not directly applicable here: the low-order stability
of~\eqref{eq-1.1} with $\tau=0$ remains to be established, while strong
approximation by the averaged system~\eqref{eq-3.22} is generally
unavailable when the slow diffusion depends on the fast
variable~\cite{Kifer2003}.

To address these issues, we introduce the smooth cutoff
$V_\delta(x)=\alpha(|x|/\delta)V(x)$, with $\alpha$ defined in
Subsection~\ref{sub312}, to remove the singularity near the origin.
We control the averaging error by freezing the slow variable and
using the mixing of the frozen fast dynamics, and the delay error
through short-time increment estimates. Choosing $\delta$ according
to the initial value yields Lyapunov estimates that can be iterated
over an infinite time horizon under~Assumption \ref{AS-3.6}, for sufficiently
small $\eps$ and $\tau$. The method also covers $\tau=0$ and slow
coefficients independent of the fast variable.

The main contributions of this paper are as follows.
\begin{enumerate}
\renewcommand{\labelenumi}{(\roman{enumi})}

\item
\emph{A novel cutoff Lyapunov method for low-order moment stability.}
We introduce an adaptive cutoff Lyapunov function to control
both averaging and delay errors for $0<p<2$, overcoming
the difficulty caused by singular derivatives near the origin.
The key step is to choose the cutoff radius according to
the initial data, yielding Lyapunov estimates that can be
iterated over an infinite time horizon without requiring
strong trajectory comparisons. This avoids the reliance on
$L^2$ comparisons with a stable reference system
in~\cite{GuoMaoYue2016,HuLiuDengMao2020}, which are not
directly available in our setting.
The method also applies to stochastic delay equations
without fast components, allowing locally Lipschitz
drift coefficients with linear growth, relaxing
the global Lipschitz requirement
in~\cite{GuoMaoYue2016,HuLiuDengMao2020} and  the bound on $\tau$
(Section~\ref{1834}).

\item
\emph{Infinite-horizon practical stability and two
distinct mechanisms for exponential stability.}
Under the averaged Lyapunov condition, we establish global
practical $p$th-moment exponential stability of the mutually
coupled system~\eqref{eq-1.1} for sufficiently small
$\eps$ and $\tau$ (Theorem~\ref{TH-3.14}).
Unlike existing averaging-based stability
results~\cite{CaoWuWu2024Functional,Liu2025,LiuKrstic2010},
which treat exogenous fast processes, our analysis allows
feedback from the slow component to the fast dynamics
and treats $0<p<2$ directly, without assuming higher-order
moment stability.
We further identify two mechanisms for $p$th-moment
and almost-sure exponential stability: homogeneous fast
dissipation yields joint slow--fast stability
(Case~I, Theorem~\ref{TH-3.16}), while vanishing sensitivity
of the slow coefficients to the fast state near the origin
yields slow-component stability even when the fast dynamics
retain persistent stochastic fluctuations
(Case~II, Theorem~\ref{TH-3.25}).

\item
\emph{High-probability exponential stability with superlinear coefficients.}
We extend the approach to locally Lipschitz systems whose
drift and diffusion coefficients may grow superlinearly,
establishing exponential stability with arbitrarily high
probability for bounded sets of initial value, provided that
$\eps$ and $\tau$ are sufficiently small
(Theorem~\ref{TH-4.13}).
For stochastic delay differential equations without fast
components, the corresponding result
(Theorem~\ref{TH-4.15}) covers, under the stated Lyapunov
conditions, nonlinear systems beyond the linear-growth
and Khasminskii-type assumptions
in~\cite{DongTangMao2022,FeiHuMaoShen2017,
GuoMaoYue2016,HuLiuDengMao2020,
LiFeiMeiFeiMao2023,LiMao2020,
ZongLiYinZhang2021};
see Remark~\ref{1631} for details.
\end{enumerate}

The remainder of the paper is organised as follows. Section \ref{SEC-3}  establishes the main stability results under global Lipschitz conditions, while Section \ref{SEC-4} extends these results to locally Lipschitz systems. Section~\ref{SEC-5} presents numerical examples. Section~\ref{SEC-6} presents the concluding remarks.

\textbf{Notations.}
Let $(\Omega,\mathcal F,\Pp)$ be a complete probability space
carrying independent Brownian motions $B^1$ and $B^2$
of dimensions $d_1$ and $d_2$, with the usual augmented
filtration $\{\mathcal F_t\}_{t\ge0}$.
For vectors and matrices, $|\cdot|$, $\|\cdot\|$ and
$\|\cdot\|_{\mathrm F}$ denote the Euclidean, operator and
Frobenius norms, respectively; $A^\top$ and
$\operatorname{tr}(A)$ denote transpose and trace.
Write $a\wedge b=\min\{a,b\}$ and
$a\vee b=\max\{a,b\}$.
For smooth functions, $\nabla f=Df$ and $D^2f$
denote the gradient and Hessian, respectively, while
$C^k(E,F)$ denotes the space of $k$-times continuously
differentiable functions.
For $\tau>0$, set $\mathcal C=C([-\tau,0],\R^d)$ with
$\|\phi\|_\infty=\sup_{-\tau\le t\le0}|\phi(t)|$.
The generic constant $C>0$ may change from line to line;
subscripted constants may depend on the indicated
parameters but are independent of sufficiently small
$\eps$ and $\tau$ unless stated otherwise.

\section{Global Lipschitz Conditions}
\label{SEC-3}

We first establish practical exponential stability
under the assumptions below and derive preliminary
estimates for subsequent analysis. We then obtain
exponential stability through two distinct mechanisms:
homogeneous fast dissipation (Case I) and practical
fast dissipation combined with vanishing fast-state
sensitivity of the slow coefficients (Case II).

\subsection{Practical exponential stability }
\label{SUBSEC-3.1}

\subsubsection{Coefficient Conditions and Preliminary Estimates}\label{2238}

\begin{assumption}\label{AS-3.1}
There exists $L\ge1$ such that, for all $x,z\in\R^d$ and $y,v\in\R^m$,
\begin{small}
\begin{equation}\label{eq-3.1}
|b(x,y)-b(z,v)|+\|\sigma(x,y)-\sigma(z,v)\|_{\rm F}
\le L\bigl(|x-z|+|y-v|\bigr).
\end{equation}
\end{small}
\end{assumption}
Set $r_0:=|b(0,0)|+\|\sigma(0,0)\|_{\rm F}$.
Assumption~\ref{AS-3.1} implies
\begin{small}
$
|b(x,y)|+\|\sigma(x,y)\|_{\rm F}
\le L(|x|+|y|)+r_0.
$
\end{small}
Moreover, setting $a(x,y):=\sigma(x,y)\sigma(x,y)^\top$,there exists $C>0$ such that
\begin{small}
\begin{equation}\label{1317}
\begin{gathered}
\|a(x,y)\|\le C(r_0+|x|+|y|)^2,\\
\|a(x,y)-a(z,v)\|
\le C(r_0+|x|+|z|+|y|+|v|)(|x-z|+|y-v|).
\end{gathered}
\end{equation}
\end{small}

We next impose dissipativity and contractivity conditions on the fast equation.

\begin{assumption}\label{AS-3.2}
The coefficients $f$ and $g$ are locally Lipschitz in $y$,
uniformly in $x$. Moreover, there exists $L_f>0$ such that,
for all $x,z\in\R^d$ and $y\in\R^m$,
\begin{small}
\begin{equation}\label{eq-3.7}
|f(x,y)-f(z,y)|+\|g(x,y)-g(z,y)\|_{\rm F}
\le L_f|x-z|.
\end{equation}
\end{small}
Furthermore, there exist $c_1,c_2,\lambda_1>0$,
$\vartheta>0$ and $\rho_d\ge0$ such that
\begin{small}
\begin{align}
2\langle y,f(x,y)\rangle+2\|g(x,y)\|_{\rm F}^2
&\le-c_1|y|^2+c_2(\rho_d^2+|x|^2),
\label{eq-3.8}\\
2\langle y-v,f(x,y)-f(x,v)\rangle
+(1+\vartheta)&\|g(x,y)-g(x,v)\|_{\rm F}^2
\le-\lambda_1|y-v|^2.
\label{eq-3.9}
\end{align}
\end{small}
\end{assumption}

Under Assumptions~\ref{AS-3.1} and \ref{AS-3.2},
\eqref{eq-1.1} admits a unique continuous global strong solution (see, e.g., \cite[Theorem~3.1]{WuYin2022}).

For each fixed $x\in\R^d$, consider the frozen fast equation
\begin{small}
\begin{equation}\label{eq-3.11}
{\rm d}Y_t^{x,y}
=f(x,Y_t^{x,y})\,{\rm d}t
+g(x,Y_t^{x,y})\,{\rm d}B_t^2,
\qquad Y_0^{x,y}=y.
\end{equation}
\end{small}
Its generator is given, for $\varphi\in C^2(\R^m,\R)$, by
\begin{small}
\begin{equation}\label{eq:fast-generator}
\mathscr A_x\varphi(y)
:=\langle f(x,y),\nabla\varphi(y)\rangle
+\frac12\operatorname{tr}\bigl(g(x,y)g(x,y)^\top D^2\varphi(y)\bigr).
\end{equation}
\end{small}

The following estimates and invariant-measure properties follow
from standard arguments; see, e.g.,
\cite[Lemmas~3.6--3.7 and Propositions~3.8--3.9]{LiuRocknerSunXie2020}.

\begin{lemma}\label{LM-3.3}
Under Assumptions~\ref{AS-3.1} and \ref{AS-3.2},
the following statements hold for some constants
$C,\lambda_2,\lambda_3>0$.
\begin{enumerate}
\item[(i)]  For each $x\in\R^d$, \eqref{eq-3.11}
admits a unique invariant measure $\mu_x$ satisfying
\begin{small}
\begin{equation}\label{eq-3.14}
\int_{\R^m}|z|^2\,\mu_x({\rm d}z)
\le C(\rho_d^2+|x|^2).
\end{equation}
\end{small}

\item[(ii)] For all $x\in\R^d$, $y,v\in\R^m$ and $t\ge0$,
\begin{small}
\begin{align}
\E|Y_t^{x,y}|^i
&\le {\rm e}^{-\lambda_2t}|y|^i
+C(\rho_d^i+|x|^i), \quad \forall i\in\{2,3\},
\label{eq-3.12}\\
\E|Y_t^{x,y}-Y_t^{x,v}|^2
&\le {\rm e}^{-\lambda_1t}|y-v|^2,
\label{eq-3.17}\\
\E|Y_t^{x,\xi_x}-Y_t^{z,\xi_z}|^2
&\le {\rm e}^{-\lambda_1t/2}\E|\xi_x-\xi_z|^2
+C|x-z|^2,\quad
\xi_x\sim\mu_x,\ \xi_z\sim\mu_z.
\label{eq--2.9}
\end{align}
\end{small}

\item[(iii)] The averaged coefficients are defined by
$\bar b(x):=\int_{\R^m}b(x,z)\,\mu_x({\rm d}z)$ and
$\bar a(x)$ $:=\int_{\R^m}a(x,z)\,\mu_x({\rm d}z)$. Then, for all $x\in\R^d$, $y\in\R^m$ and $t\ge0$,
\begin{small}
\begin{align}
\left|\E b(x,Y_t^{x,y})-\bar b(x)\right|
&\le C{\rm e}^{-\lambda_3t}(r_0+\rho_d+|x|+|y|),
\label{eq-3.15}\\
\left\|\E a(x,Y_t^{x,y})-\bar a(x)\right\|
&\le C{\rm e}^{-\lambda_3t}
\bigl((r_0+\rho_d)^2+|x|^2+|y|^2\bigr).
\label{eq-3.16}
\end{align}
\end{small}
\end{enumerate}
\end{lemma}

The following estimates follow from standard arguments; see, e.g.,
\cite[Proposition~3.8, Lemma~3.10, and the proof of Lemma~3.11]
{LiuRocknerSunXie2020}.

\begin{lemma}\label{LM-3.4}
Under Assumptions~\ref{AS-3.1}--\ref{AS-3.2},
there exists $C>0$ such that
\begin{small}
\begin{align}
|\bar b(x)|&\le C(r_0+\rho_d+|x|),\qquad
\|\bar a(x)\|\le C\bigl((r_0+\rho_d)^2+|x|^2\bigr),
\label{eq-3.24}\qquad \forall x\in\R^d,\\
|\bar b(x)-\bar b(z)|&\le C|x-z|,~~
\|\bar a(x)-\bar a(z)\|
\le C(r_0+\rho_d+|x|+|z|)|x-z|,~ \forall x,z\in\R^d.
\label{eq-3.26}
\end{align}
\end{small}
\end{lemma}

Let $\bar\sigma(x)\in\R^{d\times d}$
be the unique symmetric nonnegative
definite square root of $\bar a(x)$,
i.e., $\bar\sigma(x)\bar\sigma(x)^\top=\bar a(x)$. Let $\bar B$ be a $d$-dimensional
standard Brownian motion.
Consider the averaged system
\begin{small}
\begin{equation}\label{eq-3.22}
{\rm d}\bar X_t^{s,x}
=\bar b(\bar X_t^{s,x})\,{\rm d}t
+\bar\sigma(\bar X_t^{s,x})\,{\rm d}\bar B_t,
\qquad t\ge s,\quad \bar X_s^{s,x}=x.
\end{equation}
\end{small}
Its generator is given, for $\varphi\in C^2(\R^d,\R)$, by
\begin{small}
\begin{equation}\label{eq-3.23}
\widetilde{\mathscr L}\varphi(x)
:=\langle\bar b(x),\nabla\varphi(x)\rangle
+\frac12\operatorname{tr}\bigl(\bar a(x)D^2\varphi(x)\bigr).
\end{equation}
\end{small}

We impose the following conditions to ensure the exponential stability
 of \eqref{eq-3.22}.

\begin{assumption}\label{AS-3.5}
The averaged coefficients satisfy $\bar b(0)=0$ and $\bar a(0)=0$. This  together with Lemma~\ref{LM-3.4} and \eqref{eq--2.9} yields
$
   |\bar b(x)|\le C|x|,~~\|\bar a(x)\|\le C|x|^2. 
$
\end{assumption}

\begin{assumption}\label{AS-3.6}
There exist $p\in(0,2)$, $\lambda_6, c_-,c_+,C>0$, and a function
$V\in C(\R^d,[0,\infty))\cap C^3(\R^d\setminus\{0\},\R)$
with $V(0)=0$ such that for $j=1,2,3$,
\begin{small}
\begin{equation}\label{eq-3.29}
c_-|x|^p\le V(x)\le c_+|x|^p,
\|D^jV(x)\|\le C|x|^{p-j}, \widetilde{\mathscr L}V(x)\le-\lambda_6V(x)(x\ne0).
\end{equation}
\end{small}
\end{assumption}

Since $\bar a(0)=0$, we have
$\sigma(0,y)=0$ for $\mu_0$-almost every $y$.
Although the invariant measure $\mu_0$ need not have
full support, the following stronger condition holds
in many relevant settings and allows us to obtain
a sharper estimate of the delay error.

\begin{assumption}\label{AS-3.5a}
The diffusion coefficient satisfies
$\sigma(0,y)=0$ for all $y\in\mathbb R^m$.
\end{assumption}

\begin{remark}\label{RM-3.7}
Assumption~\ref{AS-3.6} implies the $p$th-moment exponential
stability of the averaged system \eqref{eq-3.22}, namely,
$\E|\bar X_t^{s,x}|^p
\le(c_+/c_-){\rm e}^{-\lambda_6(t-s)}|x|^p$
for all $t\ge s$.
\end{remark}

\subsubsection{Smooth Cutoff Function}\label{sub312}

Let $\alpha\in C^\infty([0,\infty))$ satisfy
$0\le\alpha\le1$, $\alpha(r)=0$ for $r\le1$,
and $\alpha(r)=1$ for $r\ge2$.
For $\delta>0$, set $V_\delta(x):=\alpha(|x|/\delta)V(x)$. The next lemma collects some useful properties of $V$ and $V_\delta$.

\begin{lemma}\label{LM-3.8}
There exists $C>0$, independent of $\delta$, such that for all $x,z\in\R^d$,
\begin{small}
\begin{align}
|V(x)-V(z)|\vee|V_\delta(x)-V_\delta(z)|
&\le C|x-z|^p,\quad \text{if }0<p\le1,
\label{eq-3.35}\\
|V(x)-V(z)|\vee|V_\delta(x)-V_\delta(z)|
&\le C(|x|^{p-1}+|z|^{p-1})|x-z|,
\quad \text{if }1<p<2,
\label{eq-3.37}\\
\|D^jV_\delta(x)\|
\le C|x|^{p-j}&\mathbf1_{\{|x|>\delta\}},
\quad j=0,1,2,3,
\label{eq-3.38}\\
0\le V(x)-V_\delta(x)\le C\delta^p,
\|D^2&V_\delta(x)\|\le C\delta^{p-2},
\|D^3V_\delta(x)\|\le C\delta^{p-3}.
\label{eq-3.40}
\end{align}
\end{small}
\end{lemma}

\begin{proof}
Set $\alpha_\delta(x)=\alpha(|x|/\delta)$.
Since $V_\delta=0$ on $\{|x|\le\delta\}$ and
$V_\delta=V$ on $\{|x|\ge2\delta\}$, the product rule,
Assumption~\ref{AS-3.6}, and
$\|D^j\alpha_\delta(x)\|\le C\delta^{-j}$
on $\delta<|x|<2\delta$ yield \eqref{eq-3.38}.
The bounds \eqref{eq-3.40} follow immediately.
Finally, \eqref{eq-3.35} and \eqref{eq-3.37}
follow from the mean-value theorem and the growth
and derivative bounds for $V$ and $V_\delta$,
by distinguishing $|x|\vee|z|\le2|x-z|$
and $|x|\vee|z|>2|x-z|$.
\end{proof}

\begin{lemma}\label{LM-3.9}
There exists $C>0$, independent of $\delta$, such that
\begin{small}
\begin{align}
&\widetilde{\mathscr L}V_\delta(x)
\le-\lambda_6V_\delta(x)+C\delta^p, \quad \forall x\in\R^d,
\label{eq-3.45}\\
&|\Psi_\delta(x,y)-\Psi_\delta(z,y)|
\le C\Lambda_\delta(y)|x-z|^p, \quad
\forall x,z\in\R^d,\ y\in\R^m, \text{if }0<p\le1,
\label{eq-3.46}\\
\text{where}\quad \Psi_\delta(x,y)
&:=\mathscr L_yV_\delta(x)-\widetilde{\mathscr L}V_\delta(x)
\nonumber=\langle b(x,y)-\bar b(x),\nabla V_\delta(x)\rangle
\nonumber\\
&\quad+\frac12\operatorname{tr}
\bigl((a(x,y)-\bar a(x))D^2V_\delta(x)\bigr),
\nonumber
\text{and}\quad \Lambda_\delta(y)
:=1+\delta^{-2}(r_0+\rho_d+|y|)^2.
\nonumber
\end{align}
\end{small}
\end{lemma}

\begin{proof}
Write $M:=r_0+\rho_d+|y|$.
By Assumptions~\ref{AS-3.1}--\ref{AS-3.2}
and Lemma~\ref{LM-3.4},
\begin{small}$$
|b(x,y)-\bar b(x)|\le C(M+|x|),\quad
\|a(x,y)-\bar a(x)\|\le C(M+|x|)^2,
$$\end{small}
and their differences at $x,z$ are bounded by
$C|x-z|$ and $C(M+|x|+|z|)|x-z|$, respectively.
Together with Lemma~\ref{LM-3.8}, this gives
\begin{small}
\begin{equation}\label{eq-3.48}
|\Psi_\delta(x,y)|
\le C|x|^{p-2}(M+|x|)^2
\le C\Lambda_\delta(y)|x|^p,
\quad |x|>\delta,
\end{equation}
\end{small}
while $\Psi_\delta(x,y)=0$ for $|x|\le\delta$. For $0<p\le1$, set $d_0:=|x-z|$ and
$R:=|x|\vee|z|$. The case $d_0=0$ is trivial.
If $R\le2d_0$, \eqref{eq-3.48} yields
$
|\Psi_\delta(x,y)-\Psi_\delta(z,y)|
\le C\Lambda_\delta(y)d_0^p.
$
Suppose $R>2d_0$. If $R\le\delta$, both terms vanish.
Otherwise $R>\delta$, and every $\xi\in[x,z]$
satisfies $|\xi|>R/2$. The mean-value theorem, and \eqref{eq-3.38} yields
\begin{small}
\[
\begin{aligned}
|\Psi_\delta(x,y)-\Psi_\delta(z,y)|
&\le Cd_0\bigl(R^{p-1}
+(M+R)R^{p-2}
+(M+R)^2R^{p-3}\bigr)\\
&\le C\Lambda_\delta(y)d_0R^{p-1}
\le C\Lambda_\delta(y)d_0^p.
\end{aligned}
\]
\end{small}
This proves \eqref{eq-3.46}. For $|x|\ge2\delta$, \eqref{eq-3.45}
follows from Assumption~\ref{AS-3.6},
while for $|x|\le\delta$ it is immediate.
On $\delta<|x|<2\delta$,
Assumption~\ref{AS-3.5} and Lemma~\ref{LM-3.8}
give
$|\widetilde{\mathscr L}V_\delta(x)|
+V_\delta(x)\le C\delta^p$,
which proves \eqref{eq-3.45}.
\end{proof}

\subsubsection{Finite-time estimates using the cutoff Lyapunov function}

Fix $T>0$ and assume $\tau\in(0,\min\{1,T/2\})$.
We first establish some finite-time estimates for \eqref{eq-1.1}.
For $s\ge0$, set
\begin{small}
\begin{equation}\label{def}
H_s:=\sup_{s-\tau\le v\le s}|X_v^\eps|,\quad
U_s:=\sup_{s-\tau\le v\le s}V(X_v^\eps), \quad
R_s^3:=H_s^3+|Y_s^\eps|^3+r_0^3+\rho_d^3.
\end{equation}
\end{small}

\begin{lemma}\label{LM-3.10}
Under Assumptions~\ref{AS-3.1}--\ref{AS-3.2},
there exists $\lambda_5>0$ such that
\begin{small}
\begin{align}
\E&\sup_{-\tau\le v\le u}|X_v^\eps|^3
\le C_T(H_0^3+\eps|Y_0^\eps|^3+r_0^3+\rho_d^3),~~
0\le u\le T, \label{eq-3.57}\\
&~\E\sup_{0\le v\le u}|Y_v^\eps|^2\leq C_T\big(1+{T}/{\eps}\big)R_0^2,\qquad
0\le u\le T,\label{eq--2.30}\\
\E|Y_u^\eps|^3&
\le C{\rm e}^{-\lambda_5u/\eps}|Y_0^\eps|^3
+C_T(H_0^3+\eps|Y_0^\eps|^3+r_0^3+\rho_d^3),
0\le u\le T, \label{eq-3.58}\\
&\E\sup_{r\le v\le r+u}|X_v^\eps-X_r^\eps|^3
\le C_Tu^{3/2}R_0^3,
0\le r\le r+u\le T, \label{eq-3.59}\\
\E|Y_T^\eps|^p&
\le C_TU_0+d_\eps(T)|Y_0^\eps|^p+C_T(r_0^p+\rho_d^p),~~
d_\eps(T):=C_T\eps^{p/3}
+C{\rm e}^{-\lambda_5pT/(3\eps)}. \label{eq-3.60}
\end{align}
\end{small}
\end{lemma}

\begin{proof}
We first establish the third-order moment bounds.
By It\^o's formula, Assumption~\ref{AS-3.2},
and Young's inequality, choosing $\lambda_5>0$
sufficiently small yields
\begin{small}
\begin{equation}\label{eq-3.62}
\E|Y_t^\eps|^3
\le C{\rm e}^{-\lambda_5t/\eps}|Y_0^\eps|^3
+C\rho_d^3
+ C/\eps\int_0^t
{\rm e}^{-\lambda_5(t-r)/\eps}
\E|X_r^\eps|^3\,{\rm d}r.
\end{equation}
\end{small}
Using linear growth condition,
H\"older's inequality, and BDG inequality gives
\begin{small}
\[
\E\sup_{-\tau\le v\le t}|X_v^\eps|^3
\le C_T(H_0^3+r_0^3)
+C_T\int_0^t
\left(\E\sup_{-\tau\le v\le r}|X_v^\eps|^3
+\E|Y_r^\eps|^3\right){\rm d}r.
\]
\end{small}
Integrating \eqref{eq-3.62} over $[0,t]$
and applying Fubini's theorem, we obtain
\begin{small}
\[
\int_0^t\E|Y_r^\eps|^3\,{\rm d}r
\le C\eps|Y_0^\eps|^3+C_T\rho_d^3
+C\int_0^t\E\sup_{-\tau\le v\le r}|X_v^\eps|^3\,{\rm d}r.
\]
\end{small}
Gronwall's inequality now proves \eqref{eq-3.57}. Similarly, by \eqref{eq-3.8}, \eqref{eq-3.57} and
the BDG inequality, we obtain \eqref{eq--2.30}.
Substituting \eqref{eq-3.57} into \eqref{eq-3.62}
gives \eqref{eq-3.58}. H\"older's inequality, \eqref{eq-3.58},
and  $c_-H_0^p\le U_0\le c_+H_0^p$ yields
\begin{small}
\[
\begin{aligned}
\E|Y_T^\eps|^p
\le\bigl(\E|Y_T^\eps|^3\bigr)^{p/3}
\le C_TU_0+
\bigl(C_T\eps^{p/3}
+C{\rm e}^{-\lambda_5pT/(3\eps)}\bigr)|Y_0^\eps|^p
+C_T(r_0^p+\rho_d^p),
\end{aligned}
\]
\end{small}
which proves \eqref{eq-3.60}. Finally, by the BDG
inequality and the growth condition,
\begin{small}
\[
\begin{aligned}
\E\sup_{r\le v\le r+u}|X_v^\eps-X_r^\eps|^3
&\le C_T(u^2+u^{1/2})
\int_r^{r+u}
\E\bigl(r_0^3+|X_q^\eps|^3+|X_{q-\tau}^\eps|^3+|Y_q^\eps|^3\bigr)
\,{\rm d}q\\
&\le C_T(u^3+u^{3/2})R_0^3
\le C_Tu^{3/2}R_0^3,
\end{aligned}
\]
\end{small}
where we used \eqref{eq-3.57} and \eqref{eq-3.58}. This proves \eqref{eq-3.59}.
\end{proof}

Fix an averaging step size $\Delta\in(0,\min\{1,T\}]$.
For each $\ell\in(0,T]$, set
$N:=\lceil\ell/\Delta\rceil-1$,
$q_j:=j\Delta$, and
$I_j:=[q_j,q_{j+1}\wedge\ell]$ for $j=0,\ldots,N$.
Then
\begin{small}
$N+1=\left\lceil\frac{\ell}{\Delta}\right\rceil
\le\frac{\ell+\Delta}{\Delta}$.
\end{small}
On each $I_j$, define the frozen fast process
$\widehat Y^j$ by
\begin{small}
\begin{equation}\label{eq-3.68}
{\rm d}\widehat Y_r^j
=\frac1\eps f(X_{q_j}^\eps,\widehat Y_r^j)\,{\rm d}r
+\frac1{\sqrt\eps}g(X_{q_j}^\eps,\widehat Y_r^j)\,{\rm d}B_r^2,
\qquad \widehat Y_{q_j}^j=Y_{q_j}^\eps.
\end{equation}
\end{small}

\begin{lemma}\label{LM-3.11}
For each $j=0,\ldots,N$ and $r\in I_j$,
\begin{small}
\begin{align}
\E|Y_r^\eps-\widehat Y_r^j|^2
\le C_T(r-q_j)R_0^2, \quad \E|\widehat Y_r^j|^3
\le C_TR_0^3.
\label{eq-3.69}
\end{align}
\end{small}
\end{lemma}

\begin{proof}
Set $Z_r:=Y_r^\eps-\widehat Y_r^j$.
By It\^o's formula, \eqref{eq-3.7},
\eqref{eq-3.9}, and Young's inequality,
\begin{small}
$$
\frac{{\rm d}}{{\rm d}r}\E|Z_r|^2
\le-\frac{\lambda_1}{2\eps}\E|Z_r|^2
+\frac C\eps\E|X_r^\eps-X_{q_j}^\eps|^2.
$$
\end{small}
Since $Z_{q_j}=0$, it follows from
\eqref{eq-3.59} and H\"older's inequality that
\begin{small}
\[
\E|Z_r|^2
\le C/\eps\int_{q_j}^r
{\rm e}^{-\lambda_1(r-u)/(2\eps)}
\E|X_u^\eps-X_{q_j}^\eps|^2\,{\rm d}u
\le C_TR_0^2(r-q_j),
\]
\end{small}
which proves the first assertion in \eqref{eq-3.69}. Moreover, by conditional \eqref{eq-3.12},
\begin{small}
\[
\E\bigl[|\widehat Y_r^j|^3\mid\mathcal F_{q_j}\bigr]
\le C{\rm e}^{-\lambda_2(r-q_j)/\eps}|Y_{q_j}^\eps|^3
+C(\rho_d^3+|X_{q_j}^\eps|^3).
\]
\end{small}
Taking expectations and using
\eqref{eq-3.57}, \eqref{eq-3.58}
gives the second assertion in \eqref{eq-3.69}.
\end{proof}

In the following, we  use the cutoff Lyapunov function to estimate
the averaging and delay errors. Henceforth, assume $R_0>0$,
since $R_0=0$ implies $X_t^\eps\equiv0$ by
\eqref{eq-3.57}, making all subsequent estimates trivial. We estimate the averaging error arising from the
Lyapunov analysis. For $\delta>0$, $\lambda\in(0,\lambda_6]$,
applying It\^o's formula
to ${\rm e}^{\lambda r}V_\delta(X_r^\eps)$, taking expectations  and using \eqref{eq-3.45} yield
\begin{small}
\begin{equation}\label{eq-3.71a}
\begin{aligned}
&{\rm e}^{\lambda\ell}\E V_\delta(X_\ell^\eps)+(\lambda_6-\lambda)
\int_0^\ell {\rm e}^{\lambda r}
\E V_\delta(X_r^\eps)\,{\rm d}r
=V_\delta(\phi(0))\\
&+\E\int_0^\ell {\rm e}^{\lambda r}
\bigl[\lambda_6 V_\delta(X_r^\eps)
+\widetilde{\mathscr L}V_\delta(X_r^\eps)\bigr]\,{\rm d}r+\E\int_0^\ell {\rm e}^{\lambda r}
\Psi_\delta(X_r^\eps,Y_r^\eps)\,{\rm d}r
+\E\mathcal D_\ell^\delta,\\
&\leq  V_\delta(\phi(0))+C_T\delta^p+\left|\E\int_0^\ell {\rm e}^{\lambda r}
\Psi_\delta(X_r^\eps,Y_r^\eps)\,{\rm d}r\right|
+|\E\mathcal D_\ell^\delta|.
\end{aligned}
\end{equation}
\end{small}
where
\begin{small}$\mathcal D_\ell^\delta
:=\frac12\int_0^\ell {\rm e}^{\lambda r}
\operatorname{tr}\bigl(
[a(X_{r-\tau}^\eps,Y_r^\eps)-a(X_r^\eps,Y_r^\eps)]
D^2V_\delta(X_r^\eps)\bigr)\,{\rm d}r$.
\end{small}

\begin{lemma}\label{LM-3.12}
For $0<\ell\le T$ and $0<\delta\le R_0$, we have
\begin{small}
\begin{equation}\label{eq-3.71}
\left|
\E\int_0^\ell {\rm e}^{\lambda r}
\Psi_\delta(X_r^\eps,Y_r^\eps)\,{\rm d}r
\right|
\le C_T\left(1+\frac{R_0^2}{\delta^2}\right)
(\ell+\Delta)
\left(\Delta^{\frac p2\wedge\frac12}
+\frac{\eps}{\Delta}\right)R_0^p.
\end{equation}
\end{small}
\end{lemma}

\begin{proof}
For $r\in I_j$, we decompose
\begin{small}
\begin{equation}\label{eq-3.72}
\begin{aligned}
\Psi_\delta(X_r^\eps,Y_r^\eps)
&=\bigl[\Psi_\delta(X_r^\eps,Y_r^\eps)
-\Psi_\delta(X_{q_j}^\eps,Y_r^\eps)\bigr]+\bigl[\Psi_\delta(X_{q_j}^\eps,Y_r^\eps)
-\Psi_\delta(X_{q_j}^\eps,\widehat Y_r^j)\bigr]
\\&+\Psi_\delta(X_{q_j}^\eps,\widehat Y_r^j)=:A_j(r)+B_j(r)+C_j(r).
\end{aligned}
\end{equation}
\end{small}
Set $h=r-q_j$ and $K_\delta:=1+R_0^2/\delta^2$.
We first estimate $A_j$. If $0<p\le1$, then
\eqref{eq-3.46}, H\"older's inequality,
and Lemma~\ref{LM-3.10} yield
\begin{small}
\[
\begin{aligned}
\E|A_j(r)|
&\le C\E|X_r^\eps-X_{q_j}^\eps|^p
+C\delta^{-2}
\bigl[\E(r_0+\rho_d+|Y_r^\eps|)^3\bigr]^{2/3}
\bigl[\E|X_r^\eps-X_{q_j}^\eps|^{3p}\bigr]^{1/3}\\
&\le C_TK_\delta h^{p/2}R_0^p.
\end{aligned}
\]
\end{small}

If $1<p<2$, set
$D_r:=|X_r^\eps-X_{q_j}^\eps|$ and
$M_r:=r_0+\rho_d+|X_r^\eps|
+|X_{q_j}^\eps|+|Y_r^\eps|$.
Writing $x=X_r^\eps$, $z=X_{q_j}^\eps$
and $y=Y_r^\eps$, we decompose
\begin{small}
\begin{align*}
A_j(r)
=&\big\langle
b(x,y)-\bar b(x)-b(z,y)+\bar b(z),
DV_\delta(x)\big\rangle+\big\langle
b(z,y)-\bar b(z),
DV_\delta(x)-DV_\delta(z)\big\rangle\\
&+\frac12\operatorname{tr}\Big[
\big(a(x,y)-\bar a(x)
-a(z,y)+\bar a(z)\big)
D^2V_\delta(x)\Big]\\
&+\frac12\operatorname{tr}\Big[
\big(a(z,y)-\bar a(z)\big)
\big(D^2V_\delta(x)
-D^2V_\delta(z)\big)\Big].
\end{align*}
\end{small}
By the coefficient difference bounds in
Assumption~\ref{AS-3.1} and
Lemma~\ref{LM-3.4}, together with
Lemma~\ref{LM-3.8} and the mean value
theorem, we obtain
\begin{small}$$
|A_j(r)|
\le CD_r\bigl(
M_r^{p-1}+\delta^{p-2}M_r
+\delta^{p-3}M_r^2
\bigr).$$\end{small}
By H\"older's inequality and Lemma~\ref{LM-3.10}, for $0\le i\le 2$,
\begin{small}$$
\E[D_rM_r^i]
\le(\E D_r^3)^{1/3}(\E M_r^{3i/2})^{2/3}
\le(\E D_r^3)^{1/3}(\E M_r^3)^{i/3}
\le C_T h^{1/2}R_0^{i+1}.$$
\end{small}
Thus
\begin{small}$
\E|A_j(r)|
\le C_T h^{1/2}
\bigl(R_0^p+\delta^{p-2}R_0^2
+\delta^{p-3}R_0^3\bigr)\le C_TK_\delta h^{1/2}R_0^p,
$\end{small}
where we used $1<p<2$ and $\delta\le R_0$. Combining both cases gives
\begin{small}
\begin{equation}\label{eq-3.73}
\E|A_j(r)|
\le C_TK_\delta h^{\frac p2\wedge\frac12}R_0^p.
\end{equation}
\end{small}

Next, we estimate $B_j(r)$. The averaged terms
in its definition cancel, giving
\begin{small}$$
B_j(r)
=\big\langle
b(X_{q_j}^\eps,Y_r^\eps)
-b(X_{q_j}^\eps,\widehat Y_r^j),
DV_\delta(X_{q_j}^\eps)\big\rangle+\frac12\operatorname{tr}\Big[
\big(a(X_{q_j}^\eps,Y_r^\eps)
-a(X_{q_j}^\eps,\widehat Y_r^j)\big)
D^2V_\delta(X_{q_j}^\eps)\Big].$$\end{small}
By the coefficient bounds and Lemma \ref{LM-3.8},
\begin{small}$$
\begin{aligned}
|B_j(r)|
\le{}&C\mathbf1_{\{|X_{q_j}^\eps|>\delta\}}
\Big[|X_{q_j}^\eps|^{p-1}+|X_{q_j}^\eps|^{p-2}
(r_0+|X_{q_j}^\eps|+|Y_r^\eps|
+|\widehat Y_r^j|)\Big]
|Y_r^\eps-\widehat Y_r^j|\\
\le{}&C\mathbf1_{\{|X_{q_j}^\eps|>\delta\}}
|X_{q_j}^\eps|^{p-2}
(r_0+|X_{q_j}^\eps|+|Y_r^\eps|
+|\widehat Y_r^j|)
|Y_r^\eps-\widehat Y_r^j|.
\end{aligned}
$$\end{small}
Consequently, since $p<2$, Cauchy--Schwarz
and Lemmas~\ref{LM-3.10}--\ref{LM-3.11} yield
\begin{small}\begin{align}
\E|B_j(r)|
&\le C\delta^{p-2}
\Big[\E(r_0+|X_{q_j}^\eps|+|Y_r^\eps|
+|\widehat Y_r^j|)^2\Big]^{1/2}
\Big[\E|Y_r^\eps-\widehat Y_r^j|^2\Big]^{1/2}
\nonumber\\
&\le C_T\delta^{p-2}R_0^2h^{1/2}
\le C_TK_\delta h^{1/2}R_0^p.
\label{eq-3.75}
\end{align}\end{small}

Finally, conditionally on $\mathcal F_{q_j}$,
$\widehat Y_r^j$ has the law of the frozen fast process
at time $h/\eps$ with initial data
$(X_{q_j}^\eps,Y_{q_j}^\eps)$.
Writing $z=X_{q_j}^\eps$ and $y_0=Y_{q_j}^\eps$, we have
\begin{small}$$
\E[C_j(r)\mid\mathcal F_{q_j}]
=\big\langle
\E[b(z,\widehat Y_r^j)\mid\mathcal F_{q_j}]
-\bar b(z),DV_\delta(z)\big\rangle+\frac12\operatorname{tr}\Big[
\big(\E[a(z,\widehat Y_r^j)\mid\mathcal F_{q_j}]
-\bar a(z)\big)D^2V_\delta(z)\Big].$$
\end{small}
By the exponential-mixing estimates
\eqref{eq-3.15}--\eqref{eq-3.16}
and Lemma~\ref{LM-3.8},
\begin{small}\[
\begin{aligned}
&\big|\E[C_j(r)\mid\mathcal F_{q_j}]\big|
\le{}C{\rm e}^{-\lambda_3h/\eps}
\mathbf1_{\{|z|>\delta\}}
\Big[|z|^{p-1}(r_0+\rho_d+|z|+|y_0|)\\
&+|z|^{p-2}
(r_0+\rho_d+|z|+|y_0|)^2\Big]
\le{}C{\rm e}^{-\lambda_3h/\eps}
\mathbf1_{\{|z|>\delta\}}
|z|^{p-2}(r_0+\rho_d+|z|+|y_0|)^2.
\end{aligned}
\]\end{small}
Taking expectations and using
H\"older's inequality and Lemma~\ref{LM-3.10} yields
\begin{small}
\begin{equation}\label{eq-3.76}
\E\bigl|\E[C_j(r)\mid\mathcal F_{q_j}]\bigr|\le C_T{\rm e}^{-\lambda_3h/\eps}
\bigl(R_0^p+\delta^{p-2}R_0^2\bigr)\le C_TK_\delta
{\rm e}^{-\lambda_3h/\eps}R_0^p.
\end{equation}
\end{small}
Since ${\rm e}^{\lambda r}$ is deterministic, the
tower property and \eqref{eq-3.76} yield
\begin{small}\begin{equation}\label{1215}
\begin{aligned}
&\left|\E\sum_{j=0}^N\int_{I_j}
{\rm e}^{\lambda r}C_j(r)\,{\rm d}r\right|
=\left|\sum_{j=0}^N\int_{I_j}
{\rm e}^{\lambda r}
\E\bigl[\E[C_j(r)\mid\mathcal F_{q_j}]\bigr]
\,{\rm d}r\right|\\
&\le\sum_{j=0}^N\int_{I_j}
{\rm e}^{\lambda r}
\E\bigl|\E[C_j(r)\mid\mathcal F_{q_j}]\bigr|
\,{\rm d}r\le C_TK_\delta R_0^p
\sum_{j=0}^N\int_{I_j}
{\rm e}^{-\lambda_3(r-q_j)/\eps}\,{\rm d}r\le C_TK_\delta(N+1)\eps R_0^p.
\end{aligned}\end{equation}\end{small}
Here we used $r\le T$ and
$\int_{I_j}{\rm e}^{-\lambda_3(r-q_j)/\eps}
\,{\rm d}r\le\eps/\lambda_3$. Combining the decomposition
\eqref{eq-3.72} with \eqref{eq-3.73},\eqref{eq-3.75} and \eqref{1215}, and $N+1\le(\ell+\Delta)/\Delta$ yields
\begin{small}
$$
\begin{aligned}
&\left|\E\int_0^\ell
{\rm e}^{\lambda r}
\Psi_\delta(X_r^\eps,Y_r^\eps)\,{\rm d}r\right|\le
\sum_{j=0}^N\int_{I_j}
{\rm e}^{\lambda r}
\bigl(\E|A_j(r)|+\E|B_j(r)|\bigr)\,{\rm d}r+
\left|\E\sum_{j=0}^N\int_{I_j}
{\rm e}^{\lambda r}C_j(r)\,{\rm d}r\right|\\
&\le C_TK_\delta R_0^p
\left(\ell\Delta^{\frac p2\wedge\frac12}
+(N+1)\eps\right)\le C_TK_\delta(\ell+\Delta)
\left(\Delta^{\frac p2\wedge\frac12}
+\frac{\eps}{\Delta}\right)R_0^p.
\end{aligned}
$$
\end{small}
This proves \eqref{eq-3.71}.
\end{proof}

We now establish the main finite-time Lyapunov estimate.
Set
\begin{small}
\begin{equation}\label{eq-3.79}
\alpha_p:=\frac p2\wedge\frac12,\quad
\beta_p:=\frac{\alpha_p}{1+\alpha_p},\quad
\nu_p:=\frac{p\beta_p}{p+2}<\frac{p}{3},\quad \tilde{\nu}_p:=\beta_p\min\{p,1\}.
\end{equation}
\end{small}

\begin{lemma}\label{LM-3.13}
Under Assumptions~\ref{AS-3.1}--\ref{AS-3.6},
without requiring Assumption~\ref{AS-3.5a},
for all $0<\ell\le T$, $0<\eps\le1$ and
$0<\tau\le\min\{1,T/2\}$,
\begin{small}
\begin{equation}\label{eq-3.81}
\E V(X_\ell^\eps)
\le {\rm e}^{-\lambda_6\ell}U_0
+C_T(\eps^{\nu_p}+\tau^{p/4})R_0^p.
\end{equation}
\end{small}
If Assumption~\ref{AS-3.5a} also holds, given $\lambda\in(0,\lambda_6)$, 
the delay rate improves to
\begin{small}
\begin{equation}\label{eq-3.81a}
\E V(X_\ell^\eps)
\le {\rm e}^{-\lambda\ell}U_0
+C_T(\eps^{\nu_p}+\tau^{p/2})R_0^p.
\end{equation}
\end{small}
\end{lemma}

\begin{proof}

We estimate $|\E\mathcal D_\ell^\delta|$ under the basic assumptions. By \eqref{1317} and Lemma~\ref{LM-3.8},
\begin{small}$$
\begin{aligned}
|\mathcal D_\ell^\delta|
&\le C_T\int_0^\ell
\|D^2V_\delta(X_r^\eps)\|
\|a(X_{r-\tau}^\eps,Y_r^\eps)
-a(X_r^\eps,Y_r^\eps)\|\,{\rm d}r\\
&\le C_T\delta^{p-2}\int_0^\ell
(r_0+\rho_d+|X_r^\eps|
+|X_{r-\tau}^\eps|+|Y_r^\eps|)
|X_r^\eps-X_{r-\tau}^\eps|\,{\rm d}r.
\end{aligned}$$\end{small}
By H\"older's inequality,
\begin{small}
\[
\begin{aligned}
\E|\mathcal D_\ell^\delta|
\le C_T\delta^{p-2}
\sqrt{\E\int_0^\ell
(r_0+\rho_d+|X_r^\eps|
+|X_{r-\tau}^\eps|+|Y_r^\eps|)^2\,{\rm d}r
}
\sqrt{\E\int_0^\ell
|X_r^\eps-X_{r-\tau}^\eps|^2\,{\rm d}r}.
\end{aligned}
\]
\end{small}
By Lemma~\ref{LM-3.10}, the first integral
is bounded by $C_T\ell R_0^2$.
For the second, the bounded initial history
and \eqref{eq-3.57} control $r<\tau$,
while \eqref{eq-3.59} controls $r\ge\tau$.
Thus,
\begin{small}
\[
\E\int_0^\ell
|X_r^\eps-X_{r-\tau}^\eps|^2\,{\rm d}r
\le C_TR_0^2
\bigl((\ell\wedge\tau)+\tau(\ell-\tau)^+\bigr)
\le C_T\tau R_0^2.
\]
\end{small}
Consequently,
\begin{small}
$\E|\mathcal D_\ell^\delta|
\le C_T\delta^{p-2}R_0^2\sqrt{\tau\ell}.$
\end{small}
Combining this estimate, using $0\le V-V_\delta\le C\delta^p$
and substituting \eqref{eq-3.71} into \eqref{eq-3.71a} and taking $\lambda=\lambda_6$ yields
\begin{small}
$$
\E V(X_\ell^\eps)
\le {\rm e}^{-\lambda_6\ell}U_0+C_T\Bigg[
\delta^p+
\left(1+\frac{R_0^2}{\delta^2}\right)
(\ell+\Delta)
\left(\Delta^{\alpha_p}+\frac{\eps}{\Delta}\right)R_0^p
+\delta^{p-2}R_0^2\sqrt{\tau\ell}
\Bigg].$$
\end{small}
Setting $\delta=\eta R_0$, $\eta:=\eps^{\beta_p/(p+2)}\vee\tau^{1/4}$, $\Delta:=\min\{1,T\}\eps^{1/(1+\alpha_p)}$ yields \eqref{eq-3.81}.

Now suppose Assumption~\ref{AS-3.5a} also holds.
By $\sigma(0,y)=0$ and Assumption~\ref{AS-3.1},
\begin{small}
\[
\|\sigma(x,y)\|_{\rm F}\le L|x|,
\qquad
\|a(x,y)-a(z,y)\|
\le C(|x|+|z|)|x-z|.
\]
\end{small}
Set $D_r^\tau:=X_r^\eps-X_{r-\tau}^\eps$.
Using the sharper derivative estimate
\eqref{eq-3.38}, we obtain
\begin{small}
\[
\begin{aligned}
|\mathcal D_\ell^\delta|
&\le C_T\int_0^\ell
\mathbf1_{\{|X_r^\eps|>\delta\}}
|X_r^\eps|^{p-2}
(|X_r^\eps|+|X_{r-\tau}^\eps|)
|D_r^\tau|\,{\rm d}r\\
&\le C_T\int_0^\ell
\mathbf1_{\{|X_r^\eps|>\delta\}}
\left(
|X_r^\eps|^{p-1}|D_r^\tau|
+|X_r^\eps|^{p-2}|D_r^\tau|^2
\right)\,{\rm d}r.
\end{aligned}
\]
\end{small}
Recall that
\begin{small}
$
\E\int_0^\ell|D_r^\tau|^2\,{\rm d}r
\le C_T\tau R_0^2,
\sup_{0\le r\le T}\E|X_r^\eps|^3
\le C_TR_0^3.
$
\end{small}
For $0<p\le1$, since $p-1\le0$,
H\"older's inequality yields
\begin{small}
$
\E|\mathcal D_\ell^\delta|
\le C_T\left(
\delta^{p-1}R_0\sqrt{\tau}
+\delta^{p-2}R_0^2\tau
\right).
$
\end{small}
Choose $\delta=\widehat\eta R_0$, where
\begin{small}
$
\widehat\eta:=
\eps^{\beta_p/(p+2)}\vee\tau^{1/2},
\Delta:=\min\{1,T\}\eps^{1/(1+\alpha_p)}.
$
\end{small}
Using $0\le V-V_\delta\le C\delta^p$
and substituting \eqref{eq-3.71} into \eqref{eq-3.71a} and taking $\lambda=\lambda_6$ proves \eqref{eq-3.81a}.

For $1<p<2$, we consider $\lambda<\lambda_6$, by Young inequality,
\begin{small}
\begin{equation}\label{1448}
\begin{aligned}
|\mathcal D_\ell^\delta|\leq C_T\int_0^\ell
\mathbf1_{\{|X_r^\eps|>\delta\}}
|X_r^\eps|^{p-1}|D_r^\tau|{\rm d}r
+C_T\int_0^\ell\mathbf1_{\{|X_r^\eps|>\delta\}}|X_r^\eps|^{p-2}|D_r^\tau|^2{\rm d}r\\
\leq\varkappa\int_0^\ell \mathbf1_{\{|X_r^\eps|>\delta\}}
|X_r^\eps|^{p}{\rm d}r+C_{T,\varkappa}\int_0^\ell |D_r^\tau|^p{\rm d}r+C_T\delta^{p-2}\int_0^l |D_r^\tau|^2{\rm d}r.
\end{aligned}
\end{equation}\end{small}
By Lemma \ref{LM-3.8}, choose $\varkappa>0$ small enough such that
$\varkappa |x|^p\mathbf1_{|x|>\delta}\leq (\lambda_6-\lambda)V_\delta(x)+C\delta^p$. Substituting \eqref{eq-3.71}
and \eqref{1448} into \eqref{eq-3.71a}
then yields
\begin{small}
\begin{equation}
\begin{aligned}
{\rm e}^{\lambda\ell}\E V(X_\ell^\eps)
&\le U_0+C_T\Bigg[
\delta^p+
K_\delta
(\ell+\Delta)
\left(\Delta^{\alpha_p}+\frac{\eps}{\Delta}\right)R_0^p+\tau^{p/2}R_0^p+\delta^{p-2}\tau R_0^{2}\Bigg].
\end{aligned}
\end{equation}
\end{small}
 Choosing $\delta=\widehat\eta R_0$,
\begin{small}
$
\widehat\eta:=
\eps^{\beta_p/(p+2)}\vee\tau^{1/2},
\Delta:=\min\{1,T\}\eps^{1/(1+\alpha_p)}
$
\end{small}
 proves \eqref{eq-3.81a}.
\end{proof}

\begin{remark}
The cutoff radius is chosen adaptively as $\delta=\eta R_0$,
with $R_0$ determined by the initial data of the time
block. A smaller $\eta$ reduces cutoff error
but increases averaging and delay errors.
Under the basic assumptions, the choice
$\eta=\eps^{\beta_p/(p+2)}\vee\tau^{1/4}$
balances these errors and yields the delay rate $\tau^{p/4}$.
Under the additional Assumption~\ref{AS-3.5a},
the choice $\eta=\eps^{\beta_p/(p+2)}\vee\tau^{1/2}$
improves the delay rate to $\tau^{p/2}$.
Neither choice requires relative scaling between
$\eps$ and $\tau$.
\end{remark}

\subsubsection{Iterative Extension, Parameter Selection, and Practical Exponential Stability}

In this subsection, we iteratively extend the finite-time
Lyapunov estimate to all $t\ge0$ and establish
practical exponential stability. Set
\begin{small}
\[
\Theta_{p}(\eps,\tau):=
\begin{cases}
\eps^{\nu_p}+\tau^{p/2},
& \text{if Assumption~\ref{AS-3.5a} also holds},\\
\eps^{\nu_p}+\tau^{p/4},
& \text{otherwise}.
\end{cases}
\]
\end{small}

\begin{theorem}\label{TH-3.14}
Suppose Assumptions~\ref{AS-3.1}--\ref{AS-3.6} hold.
Given $\lambda\in(0,\lambda_6)$, there exist
$\eps_*,\tau_*,C>0$
such that, for $\eps\in(0,\eps_*),\tau\in(0,\tau_*)$, the solution of
\eqref{eq-1.1} satisfies
\begin{small}
\begin{align}
\E|X_t^\eps|^p
\le C{\rm e}^{-\lambda t}
\bigl(\|\phi\|_\infty^p+|y|^p\bigr)
+C\Theta_p(\eps,\tau)
(r_0^p+\rho_d^p),\label{eq-3.91}\\
\E|Y_t^\eps|^p
\le C{\rm e}^{-\lambda t}
\bigl(\|\phi\|_\infty^p+|y|^p\bigr)
+C(r_0^p+\rho_d^p), \quad t\geq 0.\label{eq-feiji}
\end{align}
\end{small}
\end{theorem}

\begin{proof}
Fix $0<\lambda<\lambda'<\lambda_6$ and choose $T>0$. Recall the definition of \begin{small}$U_s,H_s,R_s$\end{small} in \eqref{def}.
By \eqref{eq-3.29},
\begin{small}
$
R_s^p\le C\bigl(U_s+|Y_s^\eps|^p
+r_0^p+\rho_d^p\bigr).
$
\end{small}
Set $\Theta:=\Theta_p(\eps,\tau)$ and
$\rho:={\rm e}^{-\lambda T}$.
By Lemma~\ref{LM-3.13}, applied conditionally
on $\mathcal F_s$ with $\ell=T-\tau$,
\begin{small}
\begin{equation}\label{eq-3.92}
\E[V(X_{s+T-\tau}^\eps)\mid\mathcal F_s]
\le {\rm e}^{-\lambda'(T-\tau)}U_s
+C_T\Theta R_s^p.
\end{equation}
\end{small}
For any $\zeta>0$, Lemma~\ref{LM-3.8},
Young's inequality, and \eqref{eq-3.29} imply for all  $x,z\in\mathbb R^d$,
\begin{small}
$V(x)\le(1+\zeta)V(z)+C_\zeta|x-z|^p$.
\end{small}
 Consequently,
\begin{small}
\begin{equation}\label{eq-3.93}
U_{s+T}\le(1+\zeta)V(X_{s+T-\tau}^\eps)
+C_\zeta\sup_{s+T-\tau\le v\le s+T}
|X_v^\eps-X_{s+T-\tau}^\eps|^p.
\end{equation}
\end{small}
By the conditional version of \eqref{eq-3.59}
and H\"older's inequality,
\begin{small}
\begin{equation}\label{eq-3.94}
\E\left[
\sup_{s+T-\tau\le v\le s+T}
|X_v^\eps-X_{s+T-\tau}^\eps|^p
\,\middle|\,\mathcal F_s\right]
\le C_T\tau^{p/2}R_s^p.
\end{equation}
\end{small}
Since $\tau^{p/2}\le\Theta$ under either
set of assumptions, \eqref{eq-3.92}--\eqref{eq-3.94}
yield
\begin{small}
\begin{equation}\label{eq-3.95}
\E[U_{s+T}\mid\mathcal F_s]
\le\bigl((1+\zeta){\rm e}^{-\lambda'(T-\tau)}
+C_{T,\zeta}\Theta\bigr)U_s+C_{T,\zeta}\Theta
\bigl(|Y_s^\eps|^p+r_0^p+\rho_d^p\bigr).
\end{equation}
\end{small}
On the other hand, the conditional version of
\eqref{eq-3.60} gives
\begin{small}
\begin{equation}\label{eq-3.96}
\E[|Y_{s+T}^\eps|^p\mid\mathcal F_s]
\le C_TU_s+d_\eps(T)|Y_s^\eps|^p
+C_T(r_0^p+\rho_d^p).
\end{equation}
\end{small}
Define
\begin{small}
$
\mathcal W_s:=U_s+\kappa|Y_s^\eps|^p,
$
\end{small} where $\kappa=K\Theta$.
Combining \eqref{eq-3.95} and
\eqref{eq-3.96} yields
\begin{small}
\begin{equation}\label{eq-3.98}
\begin{aligned}
\E[\mathcal W_{s+T}\mid\mathcal F_s]
&\le\left[(1+\zeta){\rm e}^{-\lambda'(T-\tau)}
+C_{T,\zeta}\Theta+C_TK\Theta\right]U_s\\
&\quad+\left[C_{T,\zeta}/K+d_\eps(T)\right]
\kappa|Y_s^\eps|^p+C_{T,\zeta}(K+1)\Theta
(r_0^p+\rho_d^p).
\end{aligned}
\end{equation}
\end{small}
Note 
${\rm e}^{-\lambda' T}<{\rm e}^{-\lambda T}=\rho$. Choose $\zeta>0$ small enough and $K>0$ large enough so that
\begin{small}
$(1+\zeta){\rm e}^{-\lambda' T}<\rho, C_{T,\zeta}/K<\rho$.
\end{small} Since $\nu_p<p/3$, choose $\eps_*$,$\tau_*$
 small enough so that
\begin{small}
\begin{equation}\label{eq-3.99}
(1+\zeta){\rm e}^{-\lambda'(T-\tau_*)}
+(C_{T,\zeta}+C_TK)\Theta_*<\rho, C_{T,\zeta}/K+d_{\eps_*}(T)<\rho, \mbox{where}~\Theta_*:=\Theta_p(\eps_*,\tau_*),
\end{equation}\end{small} and $\sup\limits_{\ell\in[T,2T]}d_{\eps}(\ell)\leq C_{2T}\eps^{\frac{p}{3}}+C{\rm e}^{\frac{-\lambda_5 pT}{3\eps}}<\Theta$ for $\eps\in(0,\eps^*)$. Then \eqref{eq-3.98} yields
\begin{small}
\begin{equation}\label{eq-3.100}
\E[\mathcal W_{s+T}\mid\mathcal F_s]
\le\rho\mathcal W_s+
C_{T,\zeta}\Theta(r_0^p+\rho_d^p), \mbox{if}~~\eps\in(0,\eps^*),\tau\in(0,\tau^*).
\end{equation}
\end{small}
Taking expectations
and iterating \eqref{eq-3.100} at $t_k=kT$ gives
\begin{small}
\begin{equation}\label{eq-3.101} 
\E\mathcal W_{t_k}
\le\rho^k\mathcal W_0+
\frac{C_{T,\zeta}\Theta}{1-\rho}
(r_0^p+\rho_d^p).
\end{equation}
\end{small}
Moreover, \eqref{eq-3.29} and
the definition of $R_s$ imply
\begin{small}
$
\mathcal W_0\le
C(\|\phi\|_\infty^p+|y|^p),
R_{t_k}^p\le
C\kappa^{-1}\mathcal W_{t_k}
+C(r_0^p+\rho_d^p)\leq C\mathcal W_{t_k}/\Theta+C(r_0^p+\rho_d^p).
$
\end{small}
For  $t\in[t_k,t_{k+1}]$, the H\"older inequality,
conditional version of \eqref{eq-3.58} and Lemma~\ref{LM-3.13} yields
\begin{small}
$
\E[|X_t^\eps|^p\mid\mathcal F_{t_k}]
\le C_TU_{t_k}+C_T\Theta R_{t_k}^p\le C_T\mathcal W_{t_k}
+C_T\Theta(r_0^p+\rho_d^p).
$
\end{small}
Taking expectations, using
\eqref{eq-3.101} and \begin{small}
$\rho^k={\rm e}^{-\lambda t_k}\le{\rm e}^{\lambda T}{\rm e}^{-\lambda t}$\end{small} yields 
\begin{small}
\[
\E[|X_t^\eps|^p]
\le C_{T,\zeta}\rho^k
(\|\phi\|_\infty^p+|y|^p)
+C_{T,\zeta}\Theta(r_0^p+\rho_d^p)\le C[{\rm e}^{-\lambda t}
\bigl(\|\phi\|_\infty^p+|y|^p\bigr)
+\Theta(r_0^p+\rho_d^p)].
\]
\end{small}
This proves \eqref{eq-3.91}. For $t\in[0,T]$, \eqref{eq-feiji} follows from \eqref{eq-3.60}. For $t\in[t_k,t_{k+1}]$ with $k\geq 1$, by $d_\eps(t-t_{k-1})<\Theta$, conditional \eqref{eq-3.60} gives
\begin{small}$$
\E |Y_t^\eps|^p\leq C_{2T}\E U_{t_{k-1}}+d_\eps(t-t_{k-1})\E|Y_{t_{k-1}}^\eps|^p+C_{2T}(r_0^p+\rho_d^p)\leq C(\E\mathcal W_{t_{k-1}}+r_0^p+\rho_d^p).
$$\end{small}
Substituting \eqref{eq-3.101} into the above inequality proves \eqref{eq-feiji}.
\end{proof}

\subsection{Exponential stability}

\subsubsection{Case I: Homogeneous Fast Dissipation}

We first establish exponential stability under
a homogeneous dissipativity condition for the fast dynamics.

\begin{assumption}\label{AS-3.15}
The dissipativity condition in Assumption~\ref{AS-3.2}
holds with $\rho_d=0$. Moreover, $r_0=0$, that is,
$b(0,0)=\sigma(0,0)=0$. Under these conditions,
the dissipativity inequality also implies
$f(0,0)=g(0,0)=0$.
\end{assumption}

\begin{theorem}\label{TH-3.16}
Suppose Assumptions~\ref{AS-3.1}--\ref{AS-3.6}
and \ref{AS-3.15} hold.
For every $\lambda\in(0,\lambda_6)$, there exist
$\eps_*,\tau_*,C>0$
such that, for all $\eps\in(0,\eps_*)$
and $\tau\in(0,\tau_*)$,
\begin{small}
\begin{equation}\label{eq-3.105}
\E(|X_t^\eps|^p+|Y_t^\eps|^p)
\le C{\rm e}^{-\lambda t}
\left(\|\phi\|_\infty^p+|y|^p\right)(t\ge0),
\limsup_{t\to\infty}\frac1t\log(|X_t^\eps|+|Y_t^\eps|)
\le-\frac{\lambda}{p}~\text{a.s.}
\end{equation}
\end{small}

\end{theorem}

\begin{proof}
Under Assumption~\ref{AS-3.15}, $r_0=\rho_d=0$.
Thus the first assertion of \eqref{eq-3.105} follows directly from
Theorem~\ref{TH-3.14}. To prove the second assertion, fix
$\lambda'\in(\lambda,\lambda_6)$ and choose
$\eps_*,\tau_*$ according to
Theorem~\ref{TH-3.14}.
Fix $T>0$ and let $t_k:=kT$. By H\"older's inequality, \eqref{eq-3.101} with the conditional version of \eqref{eq-3.57},\eqref{eq--2.30} yields
\begin{small}$$
\E\sup_{t_k\le t\le t_{k+1}}(|X_t^\eps|^p+|Y_t^\eps|^p)
\le C_{T,\eps,\tau}\E R_{t_k}^p
\le C_{T,\eps,\tau}\E\mathcal W_{t_k}
\le C_{T,\eps,\tau}{\rm e}^{-\lambda't_k}
\bigl(\|\phi\|_\infty^p+|y|^p\bigr).
$$\end{small}
For $\lambda<\gamma<\lambda'$, Markov's inequality gives
\begin{small}$$
\sum_{k=1}^{+\infty}\mathbb P\left(
\sup_{t_k\le t\le t_{k+1}}(|X_t^\eps|+|Y_t^{\eps}|)
\ge{\rm e}^{-\gamma t_k/p}\right)\le \sum_{k=1}^{+\infty}C_{T,\eps,\tau}
\bigl(\|\phi\|_\infty^p+|y|^p\bigr)
{\rm e}^{-(\lambda'-\gamma)kT}<+\infty.
$$\end{small}
Using the Borel--Cantelli lemma proves the second assertion of \eqref{eq-3.105}.
\end{proof}

\subsubsection{Case II: Practical Fast Dissipation}
\label{SUBSUBSEC-3.2.2}

We next establish exponential stability under a different
mechanism. The averaged system still satisfies Assumption~\ref{AS-3.6},
but the fast dissipativity may contain a nonzero residual
$\rho_d^2+|x|^2$. To compensate for this, we require
the sensitivity of the slow coefficients to the fast state
to vanish at the slow equilibrium.

\begin{assumption}\label{AS-3.17}
There exists $L>0$ such that, for all
$x,z\in\R^d$ and $y,v\in\R^m$,
\begin{small}
\begin{align}
|b(x,y)-b(z,y)|
+\normF{\sigma(x,y)-\sigma(z,y)}
&\le L|x-z|,\label{eq-3.108}\\
b(0,y)=0,\qquad \sigma(0,y)&=0,\label{eq-3.109}\\
|b(x,y)-b(x,v)|
+\normF{\sigma(x,y)-\sigma(x,v)}
&\le L(|x|\wedge1)|y-v|.\label{eq-3.110}
\end{align}
\end{small}
\end{assumption}

By \eqref{eq-3.108}--\eqref{eq-3.109}, we obtain
$
|b(x,y)|+\normF{\sigma(x,y)}\le L|x|.
$ Furthermore,
\begin{small}\begin{align}
\|a(x,y)\|&\le C|x|^2,
\|a(x,y)-a(z,y)\|
\le C(|x|+|z|)|x-z|,\label{eq-3.113}\\
&\|a(x,y)-a(x,v)\|
\le C|x|(|x|\wedge1)|y-v|.\label{eq-3.114} 
\end{align}\end{small}

\begin{assumption}\label{AS-3.18}
The local Lipschitz condition in
Assumption~\ref{AS-3.2}, together with
\eqref{eq-3.7} and \eqref{eq-3.9}, remains
in force. We replace \eqref{eq-3.8} as for some $c_1,c_2>0,\rho_d\geq 0$,
\begin{small}
\begin{equation}\label{eq-3.115}
2\langle y,f(x,y)\rangle+\normF{g(x,y)}^2
\le-c_1|y|^2+c_2(\rho_d^2+|x|^2), \forall x\in\R^d, y\in\R^m.
\end{equation}
\end{small}
\end{assumption}

Assumption~\ref{AS-3.17} implies
Assumption~\ref{AS-3.1} with $r_0=0$.
The weaker dissipativity condition
\eqref{eq-3.115}, together with
\eqref{eq-3.7} and \eqref{eq-3.9},
still yields the frozen fast process's
second-moment bounds, contraction,
unique invariant measure $\mu_x$,
and exponential mixing estimates by
the same It\^o and coupling arguments.

Unlike Case I, no third-moment estimate is
needed. Indeed, the general growth bounds may
produce products of quadratic fast-variable terms
and slow increments in the freezing error.
Here, \eqref{eq-3.108}--\eqref{eq-3.114}
remove this dependence from the slow-variable
difference estimates, while the remaining
fast-variable errors can be controlled using
second moments.

The averaged coefficients $\bar b$, $\bar a$ and
the generator $\widetilde{\mathscr L}$ are defined
as in Subsection~\ref{2238}.
Using similar argument in Lemma~\ref{LM-3.4} and
\eqref{eq-3.108}--\eqref{eq-3.114} yields that
\begin{small}
\begin{align}
|\bar b(x)|&\le C|x|,\qquad
\|\bar a(x)\|\le C|x|^2,\label{eq-3.116}\\
|\bar b(x)-\bar b(z)|&\le C|x-z|,\qquad
\|\bar a(x)-\bar a(z)\|
\le C(|x|+|z|)|x-z|.\label{eq-3.117}
\end{align}
\end{small} In particular, \eqref{eq-3.109} implies
$\bar b(0)=\bar a(0)=0$, so Assumption~\ref{AS-3.5}
holds automatically. We continue to impose
the small-$p$ Lyapunov condition in
Assumption~\ref{AS-3.6}.

We retain the cutoff function $V_\delta$ and the
fluctuation $\Psi_\delta$ defined previously.
All estimates in Lemma~\ref{LM-3.8}, together with
\eqref{eq-3.45}, remain valid.

Define $q(r):=r/(1+r)$ for $r\ge0$.
Then
\begin{small}
\begin{equation}\label{eq-3.120}
\frac1{\sqrt2}(|x|\wedge1)
\le\sqrt{q(|x|^2)}
\le |x|\wedge1.
\end{equation}
\end{small}

\begin{lemma}\label{LM-3.19}
Under Assumptions~\ref{AS-3.17}--\ref{AS-3.18},
there exists $C>0$ such that
\begin{small}
\begin{align}
|\Psi_\delta(x,y)-\Psi_\delta(z,y)|
&\le C
\begin{cases}
|x-z|^p,&0<p\le1,\\
(|x|^{p-1}+|z|^{p-1})|x-z|,&1<p<2,
\end{cases}
\label{eq-3.121}\\
|\Psi_\delta(x,y)-\Psi_\delta(x,v)|
&\le C\mathbf1_{\{|x|>\delta\}}
|x|^{p-1}\sqrt{q(|x|^2)}|y-v|,
\quad \forall p\in(0,2).
\label{eq-3.122}
\end{align}
\end{small}
\end{lemma}

\begin{proof}
By \eqref{eq-3.108}--\eqref{eq-3.113},
\eqref{eq-3.116}--\eqref{eq-3.117},
and \eqref{eq-3.38},
\begin{small}
$|\Psi_\delta(x,y)|\le C|x|^p$.
\end{small}
Moreover, the same coefficient bounds and
the first three derivative estimates for $V_\delta$
yield, whenever $|x|\vee|z|\le2(|x|\wedge|z|)$,
\begin{small}
$
|\Psi_\delta(x,y)-\Psi_\delta(z,y)|
\le C(|x|\vee|z|)^{p-1}|x-z|.
$
\end{small}
Indeed, this follows by separating the differences
of the coefficients and those of $DV_\delta$
and $D^2V_\delta$, using
$|b(x,y)-\bar b(x)|\le C|x|$ and
$\|a(x,y)-\bar a(x)\|\le C|x|^2$.

Set $d:=|x-z|$ and $R:=|x|\vee|z|$.
If $R\le2d$, $|\Psi_\delta(x,y)|\le C|x|^p$ gives
\begin{small}
$
|\Psi_\delta(x,y)-\Psi_\delta(z,y)|
\le CR^p\le Cd^p.$
\end{small}
If $R>2d$, then $|x|\wedge|z|\ge R/2$,
so the preceding local estimate yields
\begin{small}
$
|\Psi_\delta(x,y)-\Psi_\delta(z,y)|
\le CR^{p-1}d.$
\end{small}
For $0<p\le1$, this is bounded by $Cd^p$,
while for $1<p<2$ it is bounded by
$C(|x|^{p-1}+|z|^{p-1})d$.
The latter bound also holds when $R\le2d$.
This proves \eqref{eq-3.121}. Finally, the averaged terms cancel in
$\Psi_\delta(x,y)-\Psi_\delta(x,v)$.
Using \eqref{eq-3.38}, \eqref{eq-3.110},
and \eqref{eq-3.114}, we obtain
\begin{small}$
|\Psi_\delta(x,y)-\Psi_\delta(x,v)|
\le C\mathbf1_{\{|x|>\delta\}}
\left(|x|^{p-1}+|x|^{p-1}\right)
(|x|\wedge1)|y-v|\le C\mathbf1_{\{|x|>\delta\}}
|x|^{p-1}\sqrt{q(|x|^2)}|y-v|$\end{small},
where \eqref{eq-3.120} was used.
This proves \eqref{eq-3.122}.
\end{proof}

We retain the notation $H_t$ and $U_t$ introduced in \eqref{def}.
The following estimates follow by the same
argument as in
Lemma~\ref{LM-3.10}.

\begin{lemma}\label{LM-3.20}
For $T>0$, there exists $C_T>0$
independent of $\eps$,$\tau$, such that
\begin{small}
\begin{equation}\label{eq-3.126}
\E\sup_{-\tau\le v\le T}|X_v^\eps|^2
\le C_TH_0^2,\E\sup_{t\le v\le t+u}|X_v^\eps-X_t^\eps|^2\le C_TuH_0^2,
0\le t\le t+u\le T.
\end{equation}
\end{small}
\end{lemma}

Fix $\chi_0>L^2$ and define the Lyapunov--Krasovskii
functional $S_t$, the weighted fast energy $Z_t$,
and the combined energy $\bar R_t^p$ by
\begin{small}
\begin{equation}\label{eq-3.128}
S_t:=|X_t^\eps|^2+
\chi_0\int_{t-\tau}^t|X_r^\eps|^2\,{\rm d}r,
\qquad
Z_t:=q(S_t)|Y_t^\eps|^2, 
\qquad \bar R_t^p:=U_t+Z_t^{p/2}.
\end{equation}
\end{small}

\begin{lemma}\label{LM-3.21}
For every $\lambda\in(0,\min\{c_1,\lambda_6\})$,
there exists $\eps_1>0$ such that
\begin{small}
\begin{equation}\label{eq-3.129}
\E Z_t
\le {\rm e}^{-\lambda t/\eps}Z_0+C_TH_0^2,
\qquad 0\le t\le T, \mbox{for}~~ \eps\in(0,\eps_1).
\end{equation}
\end{small}
\end{lemma}

\begin{proof}
Since $B^1$ and $B^2$ are independent,
It\^o's formula gives
\begin{small}
\begin{equation}\label{eq-3.130}
\begin{aligned}
&{\rm d}Z_t
=|Y_t^\eps|^2q'(S_t)
\Bigl[2\langle X_t^\eps,b(X_t^\eps,Y_t^\eps)\rangle
+\normF{\sigma(X_{t-\tau}^\eps,Y_t^\eps)}^2
+\chi_0(|X_t^\eps|^2-|X_{t-\tau}^\eps|^2)\Bigr]{\rm d}t\\
&+2|Y_t^\eps|^2q''(S_t)
|\sigma(X_{t-\tau}^\eps,Y_t^\eps)^\top X_t^\eps|^2\,{\rm d}t
+\frac{q(S_t)}{\eps}
\Bigl[2\langle Y_t^\eps,f(X_t^\eps,Y_t^\eps)\rangle
+\normF{g(X_t^\eps,Y_t^\eps)}^2\Bigr]{\rm d}t
+{\rm d}M_t,
\end{aligned}
\end{equation}
\end{small}
where $M_t$ is a local martingale. By $q'(r)=(1+r)^{-2}$, $q''(r)\le0$,$\chi_0>L^2$, \eqref{eq-3.113} yields
\begin{small}
\begin{equation}\label{eq-3.131}
\begin{aligned}
&q'(S_t)\Bigl[
2\langle X_t^\eps,b(X_t^\eps,Y_t^\eps)\rangle
+\normF{\sigma(X_{t-\tau}^\eps,Y_t^\eps)}^2
+\chi_0(|X_t^\eps|^2-|X_{t-\tau}^\eps|^2)
\Bigr]\\
&\quad+2q''(S_t)
|\sigma(X_{t-\tau}^\eps,Y_t^\eps)^\top X_t^\eps|^2
\le\frac{(2L+\chi_0)|X_t^\eps|^2
-(\chi_0-L^2)|X_{t-\tau}^\eps|^2}
{(1+S_t)^2}
\le C_0q(S_t),
\end{aligned}
\end{equation}
\end{small}
where $C_0=2L+\chi_0$.
Combining \eqref{eq-3.115} and
\eqref{eq-3.130}--\eqref{eq-3.131}, we obtain
\begin{small}
\[
{\rm d}Z_t
\le\left[
\left(C_0-\frac{c_1}{\eps}\right)Z_t
+\frac{c_2}{\eps}q(S_t)
(\rho_d^2+|X_t^\eps|^2)
\right]{\rm d}t+{\rm d}M_t.
\]
\end{small}
By $|X_t^\eps|^2\le S_t$,
$
q(S_t)(\rho_d^2+|X_t^\eps|^2)
\le(1+\rho_d^2)S_t.
$
Fix $\lambda\in(0,\min\{c_1,\lambda_6\})$
and choose $\eps_1>0$ so that
$C_0\eps_1\le c_1-\lambda$. Using integrating factor
${\rm e}^{\lambda t/\eps}$ gives
\begin{small}
$$
\E Z_t
\le{\rm e}^{-\lambda t/\eps}Z_0+
C/\eps\int_0^t
{\rm e}^{-\lambda(t-r)/\eps}
\E S_r\,{\rm d}r
\le{\rm e}^{-\lambda t/\eps}Z_0+C_TH_0^2.
$$
\end{small}
The second inequality comes from Lemma~\ref{LM-3.20}. This proves \eqref{eq-3.129}.
\end{proof}

Fix $T>0$ and retain the mesh points $q_j$,
intervals $I_j$, number $N$, and frozen fast processes
$\widehat Y^j$ introduced in \eqref{eq-3.68}.
The following estimate follows by the same
synchronous-coupling argument as in Lemma~\ref{LM-3.11}.

\begin{lemma}\label{LM-3.22}
For every $T>0$, there exists $C_T>0$ such that,
for $r\in I_j$,
\begin{small}
\begin{equation}\label{eq-3.136}
\E|Y_r^\eps-\widehat Y_r^j|^2
\le C_T(r-q_j)\bar R_0^2.
\end{equation}
\end{small}
\end{lemma}

We retain the It\^o decomposition \eqref{eq-3.71a} and obtain the following result.

\begin{lemma}\label{LM-3.23}
For $T>0$ and 
$\bar\lambda\in(0,\lambda_6)$,$\ell\in[0,T]$, setting $\alpha_p:=\frac{1\wedge p}{2}$ and
$\delta:=\eta\bar R_0$ with $\eta\in(0,1]$,
$\Delta\in(0,\min\{1,T\})$,$\eps\in(0,\eps_1)$,
there exists $C_T>0$ such that
\begin{small}
\begin{equation}\label{eq-3.138}
\left|
\E\int_0^\ell {\rm e}^{\bar \lambda r}
\Psi_{\delta}(X_r^\eps,Y_r^\eps)\,{\rm d}r
\right|
\le C_T(1+\eta^{p-1})(\ell+\Delta)
\left(\Delta^{\alpha_p}+\frac{\eps}{\Delta}\right)
\bar R_0^p.
\end{equation}
\end{small}
\end{lemma}

\begin{proof}
It suffices to consider  $\bar R_0>0$. For $r\in I_j$, set $h=r-q_j$ and decompose
\begin{small}
\begin{equation}\label{eq-3.137}
\begin{aligned}
\Psi_{\delta}(X_r^\eps,Y_r^\eps)
&=\bigl[\Psi_{\delta}(X_r^\eps,Y_r^\eps)
-\Psi_{\delta}(X_{q_j}^\eps,Y_r^\eps)\bigr]+\bigl[\Psi_{\delta}(X_{q_j}^\eps,Y_r^\eps)
-\Psi_{\delta}(X_{q_j}^\eps,\widehat Y_r^j)\bigr]\\
&+\Psi_{\delta}(X_{q_j}^\eps,\widehat Y_r^j)=:A_j(r)+B_j(r)+C_j(r).
\end{aligned}
\end{equation}
\end{small}
By \eqref{eq-3.121}, Lemma~\ref{LM-3.20}, the same argument
as in Lemma~\ref{LM-3.12} gives
\begin{small}
\begin{equation}\label{eq-3.139}
\E|A_j(r)|
\le C_T h^{\alpha_p}H_0^p
\le C_T h^{\alpha_p}\bar R_0^p.
\end{equation}
\end{small}
For $B_j$, \eqref{eq-3.122} and
Lemma~\ref{LM-3.22} yield
\begin{small}
\begin{equation}\label{eq-3.141}
\E|B_j(r)|
\le C_T
\begin{cases}
\delta^{p-1}h^{1/2}\bar R_0,
&0<p\le1,\\
h^{1/2}H_0^{p-1}\bar R_0,
&1<p<2,
\end{cases}
\le C_T(1+\eta^{p-1})
h^{1/2}\bar R_0^p.
\end{equation}
\end{small}

It remains to estimate $C_j$.
Conditionally on $\mathcal F_{q_j}$, couple
$\widehat Y^j$ with a stationary frozen process
$\widetilde Y^j$ having invariant law
$\mu_{X_{q_j}^\eps}$.
By the exponential contraction
of the frozen fast dynamics,
\begin{small}
\[
\E\bigl[\Psi_{\delta}
(X_{q_j}^\eps,\widetilde Y^j)
\mid\mathcal F_{q_j}\bigr]=0,~~
\E\bigl[|\widehat Y_r^j-\widetilde Y^j|
\mid\mathcal F_{q_j}\bigr]
\le C{\rm e}^{-\lambda_3h/\eps}
\bigl(|Y_{q_j}^\eps|+\rho_d+|X_{q_j}^\eps|\bigr).
\]
\end{small}
Thus, by \eqref{eq-3.122},
\begin{small}
\begin{equation}\label{eq-3.145}
|\E[C_j(r)\mid\mathcal F_{q_j}]|
\le C{\rm e}^{-\lambda_3h/\eps}
\mathbf1_{\{|X_{q_j}^\eps|>\delta\}}
|X_{q_j}^\eps|^{p-1}
\sqrt{q(|X_{q_j}^\eps|^2)}\times
\bigl(|Y_{q_j}^\eps|+\rho_d+|X_{q_j}^\eps|\bigr).
\end{equation}
\end{small}
Since $q(|X_t^\eps|^2)\le q(S_t)$,
Lemma~\ref{LM-3.21} gives
\begin{small}
\[
\E\bigl[\sqrt{q(|X_{q_j}^\eps|^2)}
|Y_{q_j}^\eps|\bigr]
\le(\E Z_{q_j})^{1/2}\le\bigl({\rm e}^{-\lambda q_j/\eps}Z_0
+C_TH_0^2\bigr)^{1/2}
\le C_T\bar R_0.
\]
\end{small}
Moreover,
$\sqrt{q(|x|^2)}(\rho_d+|x|)
\le C(1+\rho_d)|x|$.
This with \eqref{eq-3.126} gives, for $0<p\le1$,
\begin{small}
\[
\E|\E[C_j(r)\mid\mathcal F_{q_j}]|
\le C_T\eta^{p-1}
{\rm e}^{-\lambda_3h/\eps}\bar R_0^p.
\]
\end{small}
For $1<p<2$, H\"older's inequality and
\eqref{eq-3.126} instead give
\begin{small}
\[
\E\bigl[
|X_{q_j}^\eps|^{p-1}
\sqrt{q(|X_{q_j}^\eps|^2)}
|Y_{q_j}^\eps|\bigr]\le
(\E|X_{q_j}^\eps|^{2p-2})^{1/2}
(\E Z_{q_j})^{1/2}
\le C_T\bar R_0^p.
\]
\end{small}
The remaining term is bounded using
$|x|^{p-1}\sqrt{q(|x|^2)}
(\rho_d+|x|)\le C(1+\rho_d)|x|^p$.
Consequently, both cases yield
\begin{small}
\begin{equation}\label{eq-3.152}
\E|\E[C_j(r)\mid\mathcal F_{q_j}]|
\le C_T(1+\eta^{p-1})
{\rm e}^{-\lambda_3h/\eps}\bar R_0^p.
\end{equation}
\end{small}
By  $\int_{I_j}{\rm e}^{-\lambda_3(r-q_j)/\eps}
\,{\rm d}r\le C\eps$,
combining \eqref{eq-3.139}, \eqref{eq-3.141},
and \eqref{eq-3.152} gives
\begin{small}
$$
\begin{aligned}
&\left|\E\int_0^\ell
{\rm e}^{\bar \lambda r}
\Psi_\delta(X_r^\eps,Y_r^\eps)\,{\rm d}r\right|\le
\int_{0}^\ell
{\rm e}^{\bar \lambda r}
\bigl(\E|A_j(r)|+\E|B_j(r)|\bigr)\,{\rm d}r+
\left|\E\sum_{j=0}^N\int_{I_j}
{\rm e}^{\bar\lambda r}C_j(r)\,{\rm d}r\right|\\
&\le C_T \Delta^{\alpha_p}\bar R_0^p+C_T\ell(1+\eta^{p-1})
\Delta^{1/2}\bar R_0^p+C_T\eps\dfrac{\ell+\Delta}{\Delta}(1+\eta^{p-1})\bar R_0^p.
\end{aligned}
$$
\end{small}
This proves \eqref{eq-3.138}.
\end{proof}

Repeating the argument in Lemma~\ref{LM-3.13}, we choose
$\eta=\eps^{\beta_p}\vee\sqrt{\tau}$ and
$\Delta=\min\{1,T\}\eps^{1/(1+\alpha_p)}$
to obtain the following estimate.

\begin{lemma}\label{LM-3.24}
Recall the definition of $\tilde{\nu}_p$ in \eqref{eq-3.79}. Given $\overline{\lambda}<\lambda_6$, for $T>0$, there exists $C_{T}>0$ such that,
for $\ell\in[0,T]$, $\eps\in(0,1\wedge\eps_1)$
and $\tau\in(0,\min\{1,T/2\})$,
\begin{small}
\begin{equation}\label{eq-3.156}
\E V(X_\ell^\eps)
\le {\rm e}^{- \bar\lambda\ell}U_0
+C_{T}\Theta_p(\eps,\tau)\bar R_0^p, \mbox{where}~\Theta_p(\eps,\tau):=\eps^{\tilde{\nu}_p}+\tau^{p/2}.
\end{equation}
\end{small}
\end{lemma}

Fix $\theta\in(0,\min\{c_1,\lambda_6\})$.
By Lemma~\ref{LM-3.21} and Holder's inequality,
\begin{small}
\begin{equation}\label{eq-3.157}
\E Z_T^{p/2}
\le C_TU_0+d_\eps(T)Z_0^{p/2},
\qquad \mbox{where  } 
d_\eps(T):={\rm e}^{-\theta pT/(2\eps)}.
\end{equation}
\end{small}

\begin{theorem}\label{TH-3.25}
Suppose Assumptions~\ref{AS-3.6},
\ref{AS-3.17}, \ref{AS-3.18} hold.
For $\lambda\in(0,\lambda_6)$,
there exist $\eps_*,\tau_*,C>0$ such that
for $\eps\in(0,\eps_*)$,$\tau\in(0,\tau_*)$,the solution of
\eqref{eq-1.1} satisfies
\begin{small}
\begin{equation}\label{eq-3.159}
\E|X_t^\eps|^p
\le C{\rm e}^{-\lambda t}
\|\phi\|_\infty^p(1+|y|^p)(t\ge0),
\quad\limsup_{t\to\infty}\frac1t\log|X_t^\eps|
\le-\frac{\lambda}{p}~~\text{a.s..}\end{equation}
\end{small}
Moreover, as $t\to\infty$, the law of $Y_t^\eps$ converges weakly to the
unique invariant measure $\mu_0$ of the fast equation frozen at $x=0$.
\end{theorem}

\begin{proof}
Fix $\lambda<\lambda'<\bar \lambda<\lambda_6$,
choose $T>0$, and set
$\rho:={\rm e}^{-\lambda'T}$ and
$\Theta:=\Theta_p(\eps,\tau)$.
Writing $\ell=T-\tau$, the conditional version of Lemma~\ref{LM-3.24}
gives
\begin{small}
\begin{equation}\label{eq-3.160}
\E[V(X_{s+\ell}^\eps)\mid\mathcal F_s]
\le{\rm e}^{-\bar\lambda(T-\tau)}U_s
+C_{T}\Theta\bar R_s^p.
\end{equation}
\end{small}
For $\zeta>0$, Lemma~\ref{LM-3.8}
 implies
$V(x)\le(1+\zeta)V(z)+C_\zeta|x-z|^p$.
Hence
\begin{small}
\[
U_{s+T}\le(1+\zeta)V(X_{s+\ell}^\eps)
+C_\zeta\sup_{s+\ell\le v\le s+T}
|X_v^\eps-X_{s+\ell}^\eps|^p.
\]\end{small}
By the conditional version of
Lemma~\ref{LM-3.20} and H\"older's inequality,
\begin{small}
\begin{equation}\label{eq-3.161}
\E\left[
\sup_{s+\ell\le v\le s+T}
|X_v^\eps-X_{s+\ell}^\eps|^p
\,\middle|\,\mathcal F_s\right]
\le C_T\tau^{p/2}H_s^p
\le C_T\tau^{p/2}\bar R_s^p.
\end{equation}
\end{small}
Since $\tau^{p/2}\le\Theta$ and
$\bar R_s^p=U_s+Z_s^{p/2}$,
combining \eqref{eq-3.160} and
\eqref{eq-3.161} yields
\begin{small}
\begin{equation}\label{eq-3.162}
\E[U_{s+T}\mid\mathcal F_s]
\le\bigl((1+\zeta)
{\rm e}^{-\bar\lambda(T-\tau)}
+C_{T,\zeta}\Theta\bigr)U_s+C_{T,\zeta}\Theta Z_s^{p/2}.
\end{equation}
\end{small}
Likewise, the conditional version of
\eqref{eq-3.157} gives
\begin{small}
\begin{equation}\label{eq-3.163}
\E[Z_{s+T}^{p/2}\mid\mathcal F_s]
\le C_TU_s+d_\eps(T)Z_s^{p/2}.
\end{equation}
\end{small}
For $\kappa>0$, define
\begin{small}
$
\mathcal W_s:=U_s+\kappa Z_s^{p/2}.
$
\end{small}
By \eqref{eq-3.162} and \eqref{eq-3.163},
\begin{small}
\begin{equation}\label{eq-3.165}
\E[\mathcal W_{s+T}\mid\mathcal F_s]
\le\bigl((1+\zeta)
{\rm e}^{-\bar\lambda(T-\tau)}
+C_{T,\zeta}\Theta+\kappa C_T\bigr)U_s+\bigl(C_{T,\zeta}\Theta/\kappa
+d_\eps(T)\bigr)\kappa Z_s^{p/2}.
\end{equation}
\end{small}
Since ${\rm e}^{-\bar\lambda T}<{\rm e}^{-\lambda' T}=\rho$.
Choose $\zeta,\kappa>0$  small enough such that
 $(1+\zeta){\rm e}^{-\overline{\lambda} T}+\kappa C_T<\rho$. We  choose $\eps_*$
and $\tau_*$
sufficiently small such that
$\Theta_*:=\Theta_p(\eps_*,\tau_*)$ satisfy
\begin{small}
\begin{equation}\label{eq-3.166}
(1+\zeta){\rm e}^{-\bar\lambda(T-\tau_*)}
+C_{T,\zeta}\Theta_*+\kappa C_T\le\rho,\quad
C_{T,\zeta}\Theta_*/\kappa+d_{\eps_*}(T)
\le\rho.
\end{equation}
\end{small}
For $\eps\in(0,\eps_*)$ and
$\tau\in(0,\tau_*)$,  \eqref{eq-3.165} implies
\begin{small}
$
\E[\mathcal W_{s+T}\mid\mathcal F_s]
\le\rho\mathcal W_s.$\end{small} Let $t_k:=kT$. Taking expectations and
iterating  yields
\begin{small}
\begin{equation}\label{eq-3.170}
\E\mathcal W_{t_k}
\le\rho^k\mathcal W_0=\rho^k(U_0+\kappa q(S_0)^{p/2}|y|^p)\leq C\rho^k\|\phi\|_\infty^p(1+|y|^p).
\end{equation}
\end{small}
For arbitrary $t=t_k+u$, $u\in[0,T]$, H\"older's
inequality gives
\begin{small}
$\E[|X_t^\eps|^p\mid\mathcal F_{t_k}]
\le C_TH_{t_k}^p
\le C_T\mathcal W_{t_k}.$
\end{small}
Taking expectations and using
\eqref{eq-3.170}, we obtain
\begin{small}
$
\E|X_t^\eps|^p
\le C_T{\rm e}^{-\lambda t}
\|\phi\|_\infty^p(1+|y|^p).
$
\end{small}
This proves the first assertion of \eqref{eq-3.159}. Conditional Lemma~\ref{LM-3.20}
and \eqref{eq-3.170} give
\begin{small}
$\E\sup_{t_k\le t\le t_{k+1}}|X_t^\eps|^p
\le C_T\rho^k\mathcal W_0.$
\end{small} Applying Markov's inequality and the
Borel--Cantelli lemma as in
Theorem~\ref{TH-3.16} yields the second assertion of \eqref{eq-3.159}.

We next prove the weak convergence of the fast component. Assumption \ref{AS-3.18} gives the unique invariant measure $\mu_0$ with finite second moment. Let $\widehat Y^\eps$ solve the frozen equation
$
{\rm d}\widehat Y_t^\eps
=\frac{1}{\eps}f(0,\widehat Y_t^\eps)\,{\rm d}t
+\frac{1}{\sqrt{\eps}}g(0,\widehat Y_t^\eps)\,{\rm d}B_t^2
$
with $\widehat Y_0^\eps\sim\mu_0$ independent of $\mathcal F$, driven by the same $B^2$ as in \eqref{eq-1.1}.
By stationarity, $\mathcal L(\widehat Y_t^\eps)=\mu_0$ for all $t\ge0$.
Set $D_t=Y_t^\eps-\widehat Y_t^\eps$. By It\^o's formula and
 Young's inequality, \eqref{eq-3.7}, and \eqref{eq-3.9} gives
\begin{small}$
{\rm d}|D_t|^2
\le\frac{1}{\eps}\bigl[-c|D_t|^2+C|X_t^\eps|^2\bigr]\,{\rm d}t
+{\rm d}M_t,$\end{small}
where $c=\lambda_1/2$, $C=L_f^2(2/\lambda_1+1+\vartheta^{-1})$,
and $M_t$ is a local martingale. Enlarge the filtration by independent initial
variable $\widehat Y_0^\varepsilon$ and denote it
by $(\mathcal G_t)_{t\ge0}$. For $u\in[0,T]$, using integrating factor ${\rm e}^{ct/\eps}$ on
$[s,s+u]$ gives
\begin{small}
\begin{equation}\label{2338}\E[|D_{u+s}|^2\mid\mathcal G_s]
\le {\rm e}^{-cu/\eps}|D_s|^2
+\frac{C}{\eps}\int_s^{s+u} {\rm e}^{-c(s+u-r)/\eps}
\E[|X_{r}^\eps|^2\mid\mathcal G_s]\,{\rm d}r\le {\rm e}^{-cu/\eps}|D_s|^2+C_TH_s^2.
\end{equation}
\end{small}
Conditional Jensen's inequality gives
$
\mathbb E\bigl[|D_{s+u}|^p\mid\mathcal G_s\bigr]
\le {\rm e}^{-cpu/(2\varepsilon)}|D_s|^p+C_TU_s$ since  $c_-H_s^p\le U_s$,  Reduce $\eps_*$ if necessary so that
$\exp(-cpT/(2\eps_*))\le\rho/2$. This with \eqref{eq-3.170} gives
\begin{small}$
\E|D_{t_{k+1}}|^p
\le \frac{\rho}{2}\E|D_{t_k}|^p+C_T\rho^k\mathcal W_0.$\end{small}
Dividing by $\rho^{k+1}$ and iterating yields
$\E|D_{t_k}|^p\le C_T\rho^k(\E|D_0|^p+\mathcal W_0)\to 0$.
For $t=t_k+u$, $u\in[0,T]$, conditional Jensen's inequality of \eqref{2338} yields 
\begin{small}$
\E|D_t|^p\le C_T(\E|D_{t_k}|^p+\E\mathcal W_{t_k})$.
\end{small}
Thus $D_t\to0$ in probability. Since $\mathcal L(\widehat Y_t^\eps)=\mu_0$,
Slutsky's theorem gives $\mathcal L(Y_t^\eps)\Rightarrow\mu_0$. \end{proof}

\section{Comparison of the Cutoff Lyapunov Method and Comparison Method for stochastic delay differential equation without fast component}\label{1834}
To illustrate the distinction between the cutoff Lyapunov
method and comparison methods developed in \cite{GuoMaoYue2016,HuLiuDengMao2020}, we consider
the special case where $b(x,y)\equiv\bar b(x)$ and
$\sigma(x,y)\equiv\bar\sigma(x)$.
In this setting, system (1.1) reduces to a stochastic
delay differential equation. In such case, \eqref{eq-3.71a} becomes
\begin{small}
\begin{equation}\label{1015}
{\rm e}^{\lambda\ell}\E V_\delta(X_\ell)
=V_\delta(\phi(0))+\E\int_0^\ell{\rm e}^{\lambda r}
[\lambda V_\delta(X_r)+\widetilde{\mathscr L}V_\delta(X_r)]\,{\rm d}r
+\E\mathcal D_\ell^\delta,
\end{equation}
\end{small}
where
\begin{small}
$\mathcal D_\ell^\delta:=\frac12\int_0^\ell
{\rm e}^{\lambda r}
\operatorname{tr}([\bar a(X_{r-\tau})-\bar a(X_r)]
D^2V_\delta(X_r))\,{\rm d}r$.
\end{small}
Suppose $\bar b$ is locally Lipschitz with
$|\bar b(x)|\le L|x|$, while $\bar\sigma$ is
globally Lipschitz and $\bar\sigma(0)=0$.
Lemma~\ref{LM-3.20} remains valid with
$H_0=\|\phi\|_\infty$, and the argument of
Lemma~\ref{LM-3.13} gives, for $\lambda\in(0,\lambda_6)$,
\begin{small}
\begin{equation}\label{1028}
\E V(X_\ell)\le{\rm e}^{-\lambda\ell}V(\phi(0))
+C_{T,\lambda}\tau^{p/2}\|\phi\|_\infty^p.
\end{equation}
\end{small}
The arguments of Theorems~\ref{TH-3.14}
and \ref{TH-3.16} then yield exponential stability
in the $p$th moment and almost surely for small $\tau$. For example, consider
\begin{small}
\[
{\rm d}X_t=X_t(1+\sin(X_t^2))\,{\rm d}t
+3X_{t-\tau}\,{\rm d}B_t.
\]
\end{small}
Here, $|\bar b(x)|\le2|x|$, but
$\bar b'(x)=1+\sin(x^2)+2x^2\cos(x^2)$
is unbounded. For $V(x)=|x|^{1/2}$,
\begin{small}
$
\widetilde{\mathscr L}V(x)
=\left(-\frac58+\frac12\sin(x^2)\right)V(x)
\le-\frac18V(x)$, for $x\ne0$.
\end{small}
Hence, the solution is exponentially stable
in the $1/2$th moment and almost surely
for sufficiently small $\tau$. 
Unlike the comparison methods developed in
\cite{GuoMaoYue2016,HuLiuDengMao2020},
which require globally Lipschitz drift coefficients,
our direct Lyapunov approach only requires
local Lipschitz continuity and linear growth.
Thus, this example is not covered by those results.

To further illustrate the difference between two methods, consider
\begin{small}\begin{equation}
{\rm d}X_t^R=\overline{b}_R(X_t^R)\,{\rm d}t+3X_{t-\tau}^R\,{\rm d}B_t,
\overline{b}_R(x)\triangleq x[1+\sin(x^2\wedge R^2)], R\geq 1.
\end{equation}\end{small}
 Clearly,
\begin{small}$
|\overline{b}_R(x)|\le2|x|,
\operatorname{Lip}(\overline{b}_R)\le L_R:=2(1+R^2).
$\end{small}
Moreover, for \begin{small}$V(x)=|x|^{1/2}$,
$\mathcal L_R^0 V(x)\le-\frac18V(x)(x\neq 0)$\end{small}. Let \begin{small}$H_0=\|\phi\|_\infty, h(T,\tau)=
\tau^{1/4}\sqrt{6+2\sqrt\tau}\,
{\rm e}^{13(T+2\tau)/4}, m_\tau=3+2\sqrt{\tau},q(T,\tau)={\rm e}^{-T/8+13\tau/4}.$\end{small} By standard moment estimates,
\begin{small}\begin{equation}
\begin{aligned}
&\mathbb E|X_t^R|^2\le {\rm e}^{13t}H_0^2,\qquad
\mathbb E|X_t^R-X_{t-\tau}^R|^2
\le \tau m_\tau^2{\rm e}^{13t}H_0^2,\quad t\ge\tau,\\[-3pt]
&\mathbb E\sup_{T+\tau\le t\le T+2\tau}
|X_t^R-X_{T+\tau}^R|^{1/2}
\le h(T,\tau)H_0^{1/2}.
\end{aligned}
\end{equation}\end{small}

\emph{(i) Cutoff Lyapunov method.}
Choose the cutoff function $\alpha$ from
Subsection~\ref{sub312} such that
$0\le\alpha'\le5/2$ and $|\alpha''|\le5$,
and set
$
V_\delta(x)=\alpha(|x|/\delta)|x|^{1/2}.
$
Using the cutoff derivative bounds and It\^o's formula,
we obtain, for $t\ge\tau$,
\begin{small}$$
\frac{{\rm d}}{{\rm d}t}\mathbb EV_\delta(X_t^R)
\leq
-\frac18\mathbb EV_\delta(X_t^R)
+A\delta^{1/2}+B(\tau,t)\delta^{-1/2}H_0
+C(\tau,t)\delta^{-3/2}H_0^2.$$\end{small}
where
\begin{small}$
A:=\frac{245\sqrt2}{2},
B(\tau,t):=9b\sqrt\tau\,m_\tau {\rm e}^{13t/2},
C(\tau,t):=\frac92c\tau m_\tau^2{\rm e}^{13t},
b:=\frac14+\frac{25\sqrt2}{2},
c:=\frac{11}{4}+5\sqrt2.
$\end{small} Choosing $\delta=d\sqrt\tau H_0$, $d>0$,
and integrating over $[\tau,T+\tau]$ yield
\begin{small}
\begin{equation}
\begin{aligned}
&\mathbb E|X_{T+\tau}^R|^{1/2}
\le
\left[
q(T,\tau)+\tau^{1/4}
E_{\rm cut}(T,\tau,d)
\right]H_0^{1/2}, E_{\rm cut}(T,\tau,d)
:=\frac{9bm_\tau {\rm e}^{13\tau/2}}{\sqrt d}
J_{13/2}(T)\\&+\sqrt d\left[
\sqrt2+8A(1-{\rm e}^{-T/8})
\right]+\frac{9cm_\tau^2{\rm e}^{13\tau}}{2d^{3/2}}
J_{13}(T),J_s(T):=\frac{{\rm e}^{sT}-{\rm e}^{-T/8}}{s+1/8}.\end{aligned}
\end{equation}
\end{small}

\emph{(ii) Comparison method.}
Let $Z_t$ solve
$
{\rm d}Z_t=\overline{b}_R(Z_t)\,{\rm d}t+3Z_t\,{\rm d}B_t$, with $Z_\tau=X_\tau^R,
$
driven by the same Brownian motion.
Writing $D_t:=X_t^R-Z_t$, we obtain
\begin{small}$
\frac{\rm d}{{\rm d}t}\mathbb E|D_t|^2
\le k_R\mathbb E|D_t|^2
+18\tau m_\tau^2{\rm e}^{13t}H_0^2,
$\end{small}
where
\begin{small}$
k_R:=2L_R+18.
$\end{small}
Gronwall's inequality gives
\begin{small}$
\mathbb E|D_{T+\tau}|^2
\le\tau D_R(T,\tau)H_0^2$\end{small},where
\begin{small}$
D_R(T,\tau)
=18m_\tau^2{\rm e}^{13\tau}
\frac{{\rm e}^{k_RT}-{\rm e}^{13T}}{k_R-13}$\end{small}.
Thus, we have
\begin{small}\begin{equation}
\mathbb E|X_{T+\tau}^R|^{1/2}
\le
\left[
q(T,\tau)+\tau^{1/4}
E_{\rm comp}(T,\tau,R)
\right]H_0^{1/2}, E_{\rm comp}(T,\tau,R):=D_R(T,\tau)^{1/4}
\end{equation}\end{small}

Combining the preceding estimates yields
\begin{small}\begin{equation}\label{0121}
Q_j=q(T,\tau)+h(T,\tau)+\tau^{1/4}E_j<1,
\qquad j\in\{\mathrm{cut},\mathrm{comp}\}.
\end{equation}\end{small}
For the cutoff method, the parameters $T,d>0$
are optimized, whereas for the comparison method
$T>0$ is optimized for each $R$. The resulting sufficient delay bounds are plotted
in Figure~\ref{fig:delay-comparison}.
In particular, the cutoff criterion is independent
of $R$, while the comparison method involves
the global Lipschitz bound $L_R$. In the latter case, the threshold value tends to $0$ as $R\to+\infty$.

\begin{figure}[htbp]
    \centering
    \includegraphics[width=0.6\textwidth]{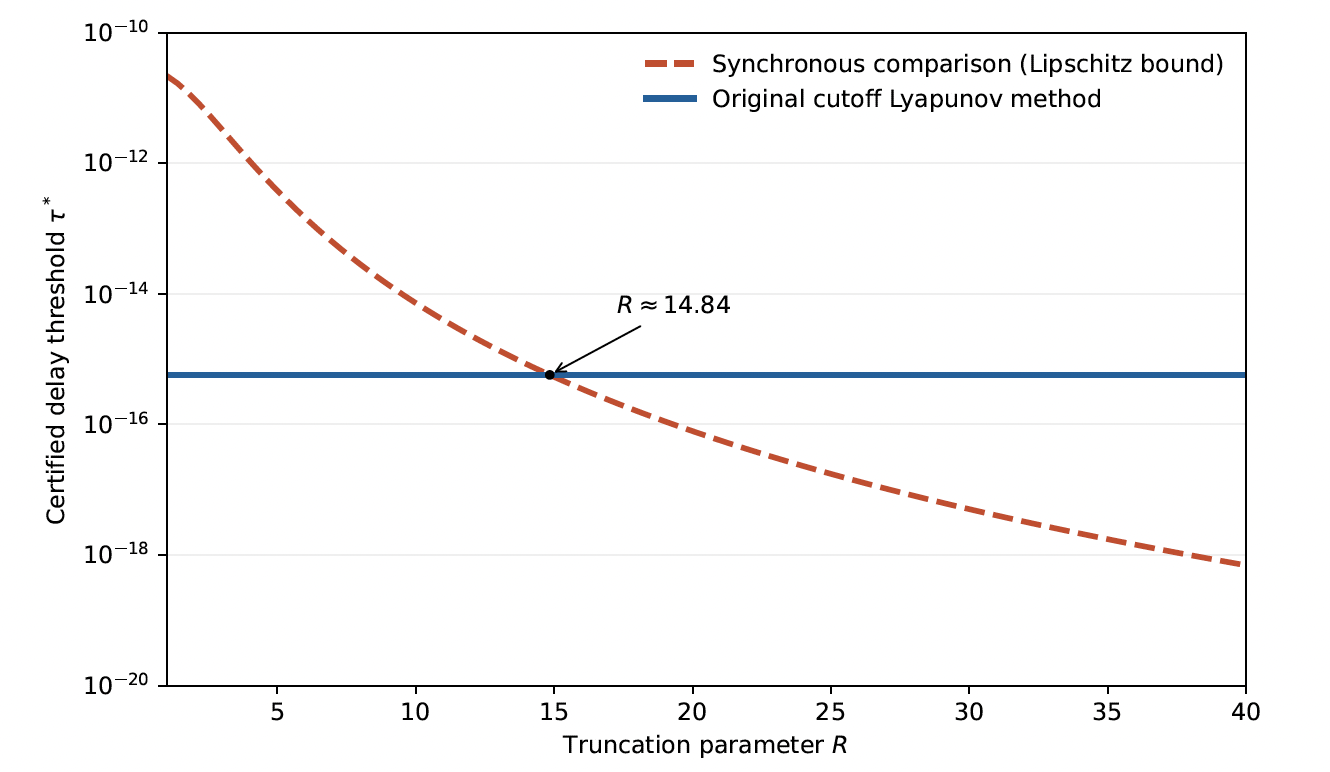}
    \caption{ Delay thresholds from \eqref{0121}. When $R>R_*\approx 14.84$, the threshold value $\tau^*$ of the cutoff Lyapunov method is larger than that of the comparison method.}
    \label{fig:delay-comparison}
\end{figure}

\section{Examples without Explicit Invariant Densities}\label{subsec:nonexplicit-invariant-examples}
Since Assumption~\ref{AS-3.6} involves the averaged coefficients, verifying it is usually very hard. We provide criteria to verify it without explicit invariant densities for the frozen fast equation.

\begin{proposition}\label{prop:certificate}
Let $p\in(0,1)$ and suppose that
$\langle x,b(x,y)\rangle\le b_*|x|^2,\forall x,y, $
where $b_*\in\mathbb R$. Suppose Assumptions~\ref{AS-3.1} and \ref{AS-3.2} hold.
Let $\mu_x$ denote the unique invariant probability measure of \eqref{eq-3.11} corresponding to ${x}\in\mathbb{R}^d$. Suppose there exist
constants $\kappa_0,\kappa_1>0$ and, for each $x\neq0$,
a function $W_x$ in the domain of the generator
$\mathscr A_x$ on $L^1(\mu_x)$ such that
\begin{equation}
\mathscr A_xW_x(y)
\ge \kappa_0|x|^2-\kappa_1
\left(
(2-p)\frac{|\sigma(x,y)^\top x|^2}{|x|^2}
-\|\sigma(x,y)\|_{\mathrm{F}}^2
\right).
\end{equation}
If $\bar b(0)=0$, $\bar a(0)=0$ and
$\kappa_0/(2\kappa_1)>b_*$, then $V(x)=|x|^p$
satisfies Assumption \ref{AS-3.6} with
$
\lambda_6=p\left(\frac{\kappa_0}{2\kappa_1}-b_*\right)>0.
$
\end{proposition}

\begin{proof}
By the invariance of $\mu_x$ and the fact that
$W_x$ belongs to the domain of $\mathscr A_x$
on $L^1(\mu_x)$, we have
$
\int\mathscr A_xW_x(y)\,\mu_x(dy)=0.
$
Integrating the generator inequality against
$\mu_x$ yields
$
(2-p)\frac{x^\top\bar a(x)x}{|x|^2}
-\operatorname{tr}(\bar a(x))
\ge\frac{\kappa_0}{\kappa_1}|x|^2.
$
Taking $V(x)=|x|^p$ yields
\begin{small}$$
\widetilde{\mathcal L}V(x)
=\frac p2|x|^{p-2}
\left(
2\langle x,\bar b(x)\rangle
+\operatorname{tr}(\bar a(x))
-(2-p)\frac{x^\top\bar a(x)x}{|x|^2}
\right)\le-p\left(\frac{\kappa_0}{2\kappa_1}-b_*\right)V(x).$$\end{small}
The inequality comes from  $\langle x,\bar b(x)\rangle\le b_*|x|^2$. This completes the proof.
\end{proof}

The following two scalar examples show how to verify the Lyapunov
condition using Proposition~\ref{prop:certificate}. The first illustrates
Theorem~\ref{TH-3.16}, and the second illustrates Theorem~\ref{TH-3.25}.

\begin{example}\label{ex:caseI}
Consider the scalar slow--fast system
\begin{equation}\label{eq:I}
\begin{cases}
 {\rm d}X_t^\eps=X_t^\eps\,{\rm d}t+
    4(Y_t^\eps-X_{t-\tau}^\eps)\,{\rm d}B_t^1,\\[2pt]
 {\rm d}Y_t^\eps=-\eps^{-1}\bigl[Y_t^\eps-X_t^\eps
      +\tanh(Y_t^\eps-X_t^\eps)\bigr]\,{\rm d}t
      +\sqrt{2/\eps}\,X_t^\eps\,{\rm d}B_t^2,
\end{cases}
\end{equation}
The coefficients are $b(x,y)=x$, $\sigma(x,y)=4(y-x)$,
$f(x,y)=-(y-x+\tanh(y-x))$, and $g(x,y)=\sqrt{2}x$.
It is easy to verify that Assumptions~\ref{AS-3.1}, \ref{AS-3.2},
and~\ref{AS-3.15} hold with $L=5$, $L_f=2+\sqrt{2}$, $c_1=1$,
$c_2=8$, $\lambda_1=2$, $\vartheta=1$, and $r_0=\rho_d=0$. By Lemma~\ref{LM-3.3}(i), the frozen fast equation has a unique invariant
measure $\mu_x$ satisfying $\int_{\R}|y|^2\,\mu_x({\rm d}y)\le Cx^2$. In particular, $\mu_0=\delta_0$.
With $z=y-x$ and $W_x(y)=|y-x|^2$, the frozen generator satisfies
$
\mathscr A_xW_x(y)
=-2z(z+\tanh z)+2x^2, |\mathscr A_xW_x(y)|\le4z^2+2x^2.
$
The second-moment bound implies that both $W_x$ and $\mathscr A_xW_x$
belong to $L^1(\mu_x)$. It$\hat{o}$'s formula and strong continuity of the
frozen semigroup on $L^1(\mu_x)$ then show that $W_x$ belongs to the domain
of its generator.
Since $z\tanh z\le z^2$, we have
$
\mathscr A_xW_x(y)\ge 2x^2-4z^2=2x^2-\tfrac14\sigma(x,y)^2.
$
Moreover, $xb(x,y)=x^2$. Since $\mu_0=\delta_0$, we have
$\bar b(0)=\bar a(0)=0$, so Assumption~\ref{AS-3.5} holds.
Taking $p=1/2$, $\kappa_0=2$, $\kappa_1=1/2$ and $b_*=1$,
we see that the generator condition in
Proposition~\ref{prop:certificate} is satisfied and
$\kappa_0/(2\kappa_1)=2>b_*$.
Hence, Proposition~\ref{prop:certificate} yields
Assumption~\ref{AS-3.6} with $V(x)=|x|^{1/2}$ and
$\lambda_6=1/2$.
By Theorem~\ref{TH-3.16}, for sufficiently small $\eps$ and $\tau$,
both the slow and fast components are exponentially stable in the
moment of order $1/2$ and almost surely.\end{example}

\begin{example}\label{ex:caseII}
Consider the scalar slow--fast system
\begin{equation}\label{eq:II}
\begin{cases}
{\rm d}X_t^\eps=X_t^\eps\,{\rm d}t
+\sigma(X_{t-\tau}^\eps,Y_t^\eps)\,{\rm d}B_t^1,\\[2pt]
{\rm d}Y_t^\eps=-\eps^{-1}\bigl[Y_t^\eps-X_t^\eps
+\tanh(Y_t^\eps-X_t^\eps)\bigr]\,{\rm d}t
+\sqrt{2/\eps}\,{\rm d}B_t^2,
\end{cases}
\end{equation}
where $\sigma(x,y)=3x+2\tanh x\,\tanh(y-x)$.
Here $b(x,y)=x$, $f(x,y)=-(y-x+\tanh(y-x))$, and $g(x,y)=\sqrt{2}$. The coefficients satisfy Assumptions~\ref{AS-3.17} and~\ref{AS-3.18} with $L=8$, $L_f=2$, $c_1=1/2$, $c_2=5$, $\lambda_1=2$, $\vartheta=1$, and $\rho_d=1$. They also satisfy Assumptions~\ref{AS-3.1} and \ref{AS-3.2}. To verify Assumption~\ref{AS-3.6}, we apply Proposition~\ref{prop:certificate} with $p=1/2$, $\kappa_0=1$, and $\kappa_1=2/9$. Setting $z=y-x$, the frozen fast equation becomes
$
{\rm d}Z_t=-(Z_t+\tanh Z_t)\,{\rm d}t+\sqrt{2}\,{\rm d}B_t.
$
Its unique invariant measure $\nu$ is symmetric. Its generator is
$
\mathcal A\phi(z)=-(z+\tanh z)\phi'(z)+\phi''(z).
$
Let $w\in D(\mathcal A)\cap L^1(\nu)$ solve the Poisson equation
$
\mathcal Aw(z)=-\tanh z.
$
Such a solution exists since $\int\tanh z\,\nu({\rm d}z)=0$ and the frozen diffusion is exponentially ergodic. Define
$
W_x(y)=\frac43x\tanh (x)\,w(y-x).
$
Then
$
\mathscr A_xW_x(y)
=-\frac43x\tanh x\,\tanh(y-x)\geq x^2-\frac19\sigma(x,y)^2,
$
where the last inequality follows from
$
x^2-\frac19\sigma(x,y)^2
=-\frac43x\tanh x\,\tanh(y-x)
-\frac49\tanh^2x\,\tanh^2(y-x).
$
Moreover, $b(0,y)=\sigma(0,y)=0$ implies $\bar b(0)=\bar a(0)=0$, and $xb(x,y)=x^2$. Thus Proposition~\ref{prop:certificate} applies with $b_*=1$.
Thus, Assumption~\ref{AS-3.6} holds with $V(x)=|x|^{1/2}$ and
$
\lambda_6=\frac58.
$
By Theorem~\ref{TH-3.25}, for sufficiently small $\eps$ and $\tau$, the slow component is exponentially stable in the moment of order $1/2$ and almost surely.
\end{example}
\section{Local Lipschitz Conditions}
\label{SEC-4}

We now extend the exponential stability result under homogeneous fast dissipation to systems whose coefficients are only locally Lipschitz in the slow variable. By combining localized Lyapunov estimates with a stopping-time argument, we establish exponential stability with arbitrarily high probability for bounded sets of initial data. As a consequence, we also obtain a small-delay stability result for stochastic differential equations without fast dynamics. 

\subsection{Coefficient Conditions and Preliminary Estimates}

\begin{assumption}\label{AS-4.1}
The coefficients $b,\sigma,f,g$ are jointly measurable.
For every $M>0$, there exists $L_M>0$ such that,
for $|x|\vee|z|\le M$ and all $y,v\in\R^m$,
\begin{small}
\begin{align}
|b(x,y)-b(z,v)|+\normF{\sigma(x,y)-\sigma(z,v)}
&\le L_M(|x-z|+|y-v|),\label{eq-4.1}\\
|f(x,y)-f(z,y)|+\normF{g(x,y)-g(z,y)}
&\le L_M|x-z|.\label{eq-4.2}
\end{align}
\end{small}
Moreover, $f,g$ are locally Lipschitz in $y$,
uniformly for $x$ in bounded sets, and there exist
$\lambda_{9,M},\vartheta_M>0$ such that,
for $|x|\le M$ and all $y,v\in\R^m$,
\begin{small}
\begin{equation}\label{eq-4.3}
2\langle y-v,f(x,y)-f(x,v)\rangle
+(1+\vartheta_M)\normF{g(x,y)-g(x,v)}^2
\le-\lambda_{9,M}|y-v|^2.
\end{equation}
\end{small}
\end{assumption}

\begin{assumption}\label{AS-4.2}
For $M>0$, there exist $c_{1,M},c_{2,M}>0$
such that \begin{small}
\begin{align}
2\langle y,f(x,y)\rangle+2\normF{g(x,y)}^2
&\le-c_{1,M}|y|^2+c_{2,M}|x|^2,\forall |x|\le M, y\in\R^m\label{eq-4.4}\\
b(0,0)=f(0,0)=g(&0,0)=\sigma(0,0)=0.\label{eq-4.5}
\end{align}
\end{small}
\end{assumption}

For each $x\in\R^d$, Eq. \eqref{eq-3.11} admits
a unique invariant measure $\mu_x$, so the averaged
coefficients and generator remain well defined.
As in Lemma~\ref{LM-3.4}, we have
\begin{small}
\begin{align}
&|\bar b(x)|\le C_M|x|,
\quad\|\bar a(x)\|\le C_M|x|^2,\forall |x|\le M,\label{eq-4.12}\\
&|\bar b(x)-\bar b(z)|\le C_M|x-z|,
\quad\|\bar a(x)-\bar a(z)\|
\le C_M(|x|+|z|)|x-z|, \forall |x|\vee|z|\le M.
\label{eq-4.13}
\end{align}
\end{small}

\begin{assumption}\label{AS-4.3}
There exist $p\in(0,2),c>0$ and
$
V\in C(\R^d,[0,\infty))\cap C^3(\R^d\setminus\{0\},\R)$ with $V(0)=0$
such that, for each $M>0$, there are positive constants
$c_{+,M},\lambda_{11,M}$ and
$C_{M}$ such that for $0<|x|\leq M,j=1,2,3$,
\begin{equation}\label{eq-4.14}
c|x|^p\le V(x)\le c_{+,M}|x|^p, 
\widetilde{\mathscr{L}}V(x)\le-\lambda_{11,M}V(x),
\|D^jV(x)\|\le C_{M}|x|^{p-j}.
\end{equation}
\end{assumption}

Fix a deterministic initial pair $(\phi,y)$. For
$M>\|\phi\|_\infty$, with $\inf\varnothing=\infty$, define
\begin{small}
\begin{equation}\label{eq-4.17}
\rho_M:=\inf\{t\ge0:|X_t^\eps|\ge M\}, \mbox{while } \rho_\infty:=\lim_{M\to\infty}\rho_M
\end{equation}
\end{small}
is the explosion time of the maximal local solution
of \eqref{eq-1.1}. By It\^o's formula,
\begin{small}
\begin{equation}\label{1522}
\begin{aligned}
&\E\!\left[
{\rm e}^{\lambda(\ell\wedge\rho_M)}
V_\delta(X_{\ell\wedge\rho_M}^\eps)
\right]
=V_\delta(\phi(0))
+\E\int_0^{\ell\wedge\rho_M}
{\rm e}^{\lambda r}
\bigl[\lambda V_\delta(X_r^\eps)
\\&\quad+\widetilde{\mathscr L}V_\delta(X_r^\eps)\bigr]
\,{\rm d}r+\E\int_0^{\ell\wedge\rho_M}
{\rm e}^{\lambda r}
\Psi_\delta(X_r^\eps,Y_r^\eps)\,{\rm d}r
+\E\mathcal D_{\ell\wedge\rho_M}^\delta,
\end{aligned}
\end{equation}
\end{small}
where 
\begin{small}
$\mathcal D_{\ell\wedge\rho_M}^\delta
:=\frac12\int_0^{\ell\wedge\rho_M}
{\rm e}^{\lambda r}
\operatorname{tr}(
[a(X_{r-\tau}^\eps,Y_r^\eps)-a(X_r^\eps,Y_r^\eps)]
D^2V_\delta(X_r^\eps))\,{\rm d}r,
$
\end{small} and $\Psi_\delta$ is defined in \eqref{eq-3.46}.

\subsection{Exponential Stability in Large Probability}

In this subsection, we establish exponential stability
in large probability for sufficiently small $\eps$ and $\tau$ (Theorems~\ref{TH-4.13} and
\ref{TH-4.15}). We retain the notation of Section~3
and, for $s<\rho_M$, set
\begin{small}
\[
H_s:=\sup_{s-\tau\le v\le s}|X_v^\eps|,
\quad U_s:=\sup_{s-\tau\le v\le s}V(X_v^\eps),
\quad R_s^3:=H_s^3+|Y_s^\eps|^3.
\]
\end{small}
Unlike the unstopped case, the indicator
$\1_{\{r<\rho_M\}}$ is not
$\cF_{q_j}$-measurable. Thus, the frozen-fast term
requires an additional estimate near the first exit time.

\begin{lemma}\label{LM-4.7}
Suppose Assumption~\ref{AS-4.1}, \ref{AS-4.2}, and
\ref{AS-4.3} hold.
For $T>0$, $0<\ell\le T$,
$0<\eta\le1$, and
$0<\Delta\le\min\{1,T\}$, set
$\delta=\eta R_0$. Then
\begin{small}
\begin{equation}\label{eq-4.39}
\begin{aligned}
&\left|\E\int_0^\ell
\1_{\{r<\rho_M\}}{\rm e}^{\lambda r}
\Psi_\delta(X_r^\eps,Y_r^\eps)\,{\rm d}r\right|\\
&\quad\le C_{M,T}R_0^p
\left[\eta^{-2}(\ell+\Delta)
\left(\Delta^{\alpha_p}+\frac{\eps}{\Delta}\right)
+\eta^{p-2}(\ell+\Delta)^{2/3}\Delta^{1/3}\right].
\end{aligned}
\end{equation}
\end{small}
\end{lemma}

\begin{proof}
The case $R_0=0$ is trivial.
Otherwise, use the three-part decomposition
\eqref{eq-3.72} with terms
$A_j(r),B_j(r),C_j(r)$, based on the
stopped slow process and its frozen fast
processes. As in Lemma~\ref{LM-3.12},
the local moment and freezing estimates give
\begin{small}
\begin{equation}\label{eq-4.40}
\E\bigl[|A_j(r)|\1_{\{r<\rho_M\}}+|B_j(r)|\1_{\{r<\rho_M\}}\bigr]
\le C_{M,T}\eta^{-2}
(r-q_j)^{\alpha_p}R_0^p,
\qquad r\in I_j.
\end{equation}
\end{small}
For the frozen term,
\begin{small}
\begin{equation}\label{eq-4.41}
\1_{\{r<\rho_M\}}C_j(r)
=\1_{\{q_j<\rho_M\}}C_j(r)
-\1_{\{q_j<\rho_M\}}
\1_{\{r\ge\rho_M\}}C_j(r).
\end{equation}
\end{small}
Since $\{q_j<\rho_M\}\in\cF_{q_j}$,
the conditional mixing argument of
Lemma~\ref{LM-3.12} yields
\begin{small}
\begin{equation}\label{eq-4.42}
\left|\E\sum_j\1_{\{q_j<\rho_M\}}
\int_{I_j}{\rm e}^{\lambda r}
C_j(r)\,{\rm d}r\right|
\le C_{M,T}\eta^{-2}
(\ell+\Delta)\frac{\eps}{\Delta}R_0^p.
\end{equation}
\end{small}
To estimate the correction term, put
$G_j:=\{q_j<\rho_M\le q_{j+1}\wedge\ell\}$.
The local coefficient bounds imply, for $|x|\leq M$,
\begin{small}
\begin{equation}\label{eq-4.43}
|\Psi_\delta(x,y)|\leq C_M(|x|+|y|)|x|^{p-1}\mathbf1_{|x|>\delta}+C_M(|x|+|y|)^2|x|^{p-2}\mathbf1_{|x|>\delta}
\le C_M\delta^{p-2}(|x|^2+|y|^2).
\end{equation}
\end{small}
The third-moment estimates, obtained as in
Lemmas~\ref{LM-3.10} and \ref{LM-3.11}
for the stopped slow process and the frozen
fast process, yield
\begin{small}
\begin{equation}\label{eq-4.44}
\E\left[
\1_{\{q_j<\rho_M\}}|C_j(r)|^{3/2}\right]
\le C_{M}\delta^{3(p-2)/2}\E[|X_{q_j}^\eps|^3+|\hat{Y}_j^\eps|^3;q_j<\rho_M]
\leq C_{M,T}\delta^{3(p-2)/2}R_0^3.
\end{equation}
\end{small}
Since $\{q_j<\rho_M\le r\}\subseteq G_j$
for $r\in I_j$, H\"older's inequality gives
\begin{small}
\begin{equation}\label{eq-4.45}
\begin{aligned}
&\hspace{1cm}\E\sum_j\1_{\{q_j<\rho_M\}}
\int_{I_j}\1_{\{r\ge\rho_M\}}
{\rm e}^{\lambda r}|C_j(r)|\,{\rm d}r\\
&\quad\le C_{M,T}\sum_j\int_{I_j}
\left(\E[\1_{\{q_j<\rho_M\}}|C_j(r)|^{3/2}]\right)^{2/3}
\Pp(G_j)^{1/3}\,{\rm d}r\le C_{M,T}\eta^{p-2}R_0^p
\Delta\sum_j\Pp(G_j)^{1/3}\\
&\quad\le C_{M,T}\eta^{p-2}R_0^p
\Delta(N+1)^{2/3}
\left(\sum_j\Pp(G_j)\right)^{1/3}\le C_{M,T}\eta^{p-2}
(\ell+\Delta)^{2/3}\Delta^{1/3}R_0^p.
\end{aligned}
\end{equation}
\end{small}
Here, the third inequality follows from
H\"older's inequality for finite sums,
while the last uses
$\sum_j\Pp(G_j)\le1$ and
$N+1\le(\ell+\Delta)/\Delta$.
Combining \eqref{eq-4.40}--\eqref{eq-4.42}
and \eqref{eq-4.45} proves \eqref{eq-4.39}.
\end{proof}

\begin{lemma}\label{LM-4.8}
Let the assumptions of Lemma~\ref{LM-4.7} hold.
Given $T>0$, $\lambda\in(0,\lambda_{11,M})$,for $\eps\in(0,1]$,
$\tau\in(0,\min\{1,T/2\}]$, and
$\ell\in(0,T]$. Then
\begin{small}
\begin{align}
\E\left[{\rm e}^{\lambda\sigma_\ell}
V(X_{\sigma_\ell}^\eps)\right]
&\le U_0+C_{M,T,\lambda}
\left(\eps^{\nu_p^{\rm I}}+\tau^{p/4}\right)R_0^p,
\quad \mbox{where }\sigma_\ell:=\ell\wedge\rho_M,\label{eq-4.51}\\
\Pp(\rho_M\le T)
&\le\left[U_0+C_{M,T,\lambda}
\left(\eps^{\nu_p^{\rm I}}+\tau^{p/4}\right)R_0^p\right]
/{cM^p}.\label{eq-4.64}
\end{align}
\end{small}
\end{lemma}

\begin{proof}
Set
\begin{small}
$
\alpha_p^{\rm I}:=\alpha_p\wedge\frac13,
\beta_p^{\rm I}:=\frac{\alpha_p^{\rm I}}{1+\alpha_p^{\rm I}},
\nu_p^{\rm I}:=\frac{p\beta_p^{\rm I}}{p+2}
$
\end{small} and
choose
\begin{small}
$
\eta:=\eps^{\beta_p^{\rm I}/(p+2)}\vee\tau^{1/4},
\Delta:=\min\{1,T\}\eps^{1/(1+\alpha_p^{\rm I})}
$
\end{small} in Lemma \ref{LM-4.7}. Then, substituting the delay estimates of Lemma \ref{LM-3.13} into \eqref{1522} yields \eqref{eq-4.51}. On $\{\rho_M\le T\}$, we have
$\sigma_T=\rho_M$ and
$|X_{\rho_M}^\eps|=M$.
Hence
$cM^p\1_{\{\rho_M\le T\}}
\le {\rm e}^{\lambda\sigma_T}
V(X_{\sigma_T}^\eps)$.
Taking expectations proves \eqref{eq-4.64}.
\end{proof}

\begin{theorem}\label{TH-4.13}
Assume Assumptions~\ref{AS-4.1}-~\ref{AS-4.3} hold.
For  $\bar R_0>0$, $\nu\in(0,1)$,
there exist $\eps_*=\eps_*(\bar R_0,\nu)>0,\tau_*=\tau_*(\bar R_0,\nu)>0$ such that,
for $\eps\in(0,\eps_*)$,
$\tau\in(0,\tau_*)$, and 
initial value satisfying
$\|\phi\|_\infty\vee|y|<\bar R_0$, the solution of
\eqref{eq-1.1} satisfies
\begin{small}
\begin{equation}\label{eq-4.67}
\Pp\left(\rho_\infty=\infty,\quad
\limsup_{t\to\infty}\frac1t\log(|X_t^\eps|+|Y_t^\eps|)
<0\right)\ge1-\nu.
\end{equation}
\end{small}
\end{theorem}

\begin{proof}
Fix $T>1$, let $t_k=kT$ and
$\mathcal S_k:=\{\rho_M>t_k\}$.
All estimates below are understood on
$\mathcal S_k$. Fix $\zeta\in(0,1)$ and put
$\ell=T-\tau\ge T/2$. Denote by $\Theta^I_{T,p}(\eps,\tau):=\eps^{\nu_p^{\rm I}}+\tau^{p/4}$. Applying Lemma~\ref{LM-4.8}
conditionally on $\cF_{t_k}$,
and using the inequality
$V(x)\le(1+\zeta)V(z)+C_{\zeta,M}|x-z|^p$
together with the stopped increment
estimate yields
\begin{small}
\begin{equation}\label{eq-4.74}
\E[\1_{\mathcal S_{k+1}}U_{t_{k+1}}
\mid\cF_{t_k}]\le\1_{\mathcal S_k}
\left[(1+\zeta)
{\rm e}^{-\lambda(T-\tau)}U_{t_k}
+C_{M,T,\zeta}\Theta^I_{T,p}(\eps,\tau)
\bigl(U_{t_k}+|Y_{t_k}^\eps|^p\bigr)\right].
\end{equation}
\end{small}
The conditional fast-moment estimate gives
\begin{small}
\begin{equation}\label{eq-4.77}
\E[\1_{\mathcal S_{k+1}}|Y_{t_{k+1}}^\eps|^p
\mid\cF_{t_k}]\le\1_{\mathcal S_k}
\left[C_{M,T}U_{t_k}
+d_{\eps,M}^{\rm I}|Y_{t_k}^\eps|^p\right],
\end{equation}
\end{small}
where
$d_{\eps,M}^{\rm I}
=C_{M,T}\eps^{p/3}
+C{\rm e}^{-\lambda_{12,M}pT/(3\eps)}
\to0$ as $\eps\downarrow0$. Choose $M>2\bar R_0$ sufficiently large that
\begin{small}
$
\frac{4(c_{+,\bar R_0}+1)\bar R_0^p}
{cM^p}\le\nu.
$
\end{small}
Next choose $T$ so that
$(1+\zeta){\rm e}^{-\lambda_{11,M} T/2}
\le1/4$, and then $\kappa\in(0,1]$
such that $\kappa C_{M,T}\le1/8$.
Finally, choose $\eps_*$ and $\tau_*$
sufficiently small that
\begin{small}
$
C_{M,T,\zeta}
\Theta_{T,p}^I(\eps_*,\tau_*)
\le\frac{\kappa}{8},
d_{\eps,M}^{\rm I}\le\frac14
\quad(0<\eps\le\eps_*),
$
$
C_{M,T}\Theta^I_{T,p}(\eps_*,\tau_*)\le\kappa.
$
\end{small}
Define
\begin{small}
$
\mathcal W_t^{\rm I}:=
U_t+\kappa|Y_t^\eps|^p.
$
\end{small}
Combining \eqref{eq-4.74} and
\eqref{eq-4.77}, we obtain
\begin{small}
\begin{equation}\label{eq-4.81}
\E[\1_{\mathcal S_{k+1}}
\mathcal W_{t_{k+1}}^{\rm I}
\mid\cF_{t_k}]
\le\frac12\1_{\mathcal S_k}
\mathcal W_{t_k}^{\rm I}.
\end{equation}
\end{small}
Taking expectation yields $\E[\1_{\mathcal S_k}\mathcal W_{t_k}^{\rm I}]
\le2^{-k}\mathcal W_0^{\rm I}$. Conditional version of \eqref{eq-4.64} gives 
\begin{small}
\begin{equation}\label{eq-4.82}
\Pp(\mathcal S_k\setminus\mathcal S_{k+1}
\mid\cF_{t_k})
\le\frac{2}{cM^p}
\1_{\mathcal S_k}\mathcal W_{t_k}^{\rm I}.
\end{equation}
\end{small}
Let $\mathcal S_\infty
:=\bigcap_{k\ge0}\mathcal S_k
=\{\rho_M=\infty\}$.
By \eqref{eq-4.82},
\begin{small}
\begin{equation}\label{eq-4.84}
\Pp(\mathcal S_\infty^c)
\le\frac2{cM^p}
\sum_{k=0}^\infty2^{-k}\mathcal W_0^{\rm I}=\frac{4\mathcal W_0^{\rm I}}
{cM^p}\le\nu.
\end{equation}
\end{small}
Using $\E[\1_{\mathcal S_k}\mathcal W_{t_k}^{\rm I}]
\le2^{-k}\mathcal W_0^{\rm I}$, the conditional
maximal-moment estimate gives
\begin{small}
$\E\left[\1_{\mathcal S_\infty}
\sup_{t_k\le t\le t_{k+1}}
(|X_t^\eps|^p+|Y_t^\eps|^p)\right]
\le C_{M,T,\kappa,\eps,\tau}2^{-k}
\mathcal W_0^{\rm I}$.
\end{small}
Combining the Borel--Cantelli lemma with
$\mathcal S_\infty\subseteq
\{\rho_\infty=\infty\}$ and
\eqref{eq-4.84} proves \eqref{eq-4.67}.
\end{proof}


\begin{remark}
The key idea is to estimate the exit probability
on each time block using geometric contraction.
These probabilities form a geometrically decaying
series whose sum can be made smaller than $\nu$.
On the non-exit event, almost sure exponential
stability follows from the same contraction.
\end{remark}

Without fast dynamics, Theorem~\ref{TH-4.13}
immediately yields the following result.

\begin{theorem}\label{TH-4.15}
Suppose that $b(x,y)\equiv\bar b(x)$ and
$\sigma(x,y)\equiv\bar\sigma(x)$ are locally
Lipschitz with $\bar b(0)=0, \bar{\sigma}(0)=0$, and Assumption~\ref{AS-4.3} holds.
For any $\hat R_0>0$ and $\nu\in(0,1)$, there exists
$\tau_*=\tau_*(\hat R_0,\nu)>0$ such that, for all $0<\tau\le\tau_*$
and the initial value satisfying
$\|\phi\|_\infty<\hat R_0$, the solution of
\eqref{eq-1.1} satisfies
\begin{small}
\begin{equation}\label{eq-4.90}
\Pp\left(\rho_\infty=\infty,\quad
\limsup_{t\to\infty}\frac1t\log|X_t|<0\right)
\ge1-\nu.
\end{equation}
\end{small}
\end{theorem}

\begin{remark}\label{1631}
Theorem~\ref{TH-4.15} applies to a broader class
of nonlinear SDEs with superlinear coefficients.
For example, consider
$\bar b(x)=x^3+x$ and $\bar\sigma(x)=x^5+3x$,
which violate the linear growth condition in
\cite{GuoMaoYue2016,HuLiuDengMao2020}
and the Khasminskii-type conditions in
\cite{DongTangMao2022,FeiHuMaoShen2017,
LiFeiMeiFeiMao2023,LiMao2020,
ZongLiYinZhang2021}.
Taking $V(x)=|x|^{1/2}$ yields
$\widetilde{\mathscr{L}}V(x)\le -\frac{13}{24}V(x)$
for $x\ne0$.
Thus, Theorem~\ref{TH-4.15} guarantees
exponential stability with arbitrarily high
probability for sufficiently small delays,
despite the superlinear growth of both coefficients.
\end{remark}

\section{Numerical examples}\label{SEC-5}

In this section, we illustrate the stability results obtained in Sections~\ref{SEC-3} and~\ref{SEC-4}.

\subsection{Two Exponential Stability Mechanisms}

\begin{figure}[htbp]
    \centering
    \includegraphics[width=\textwidth]{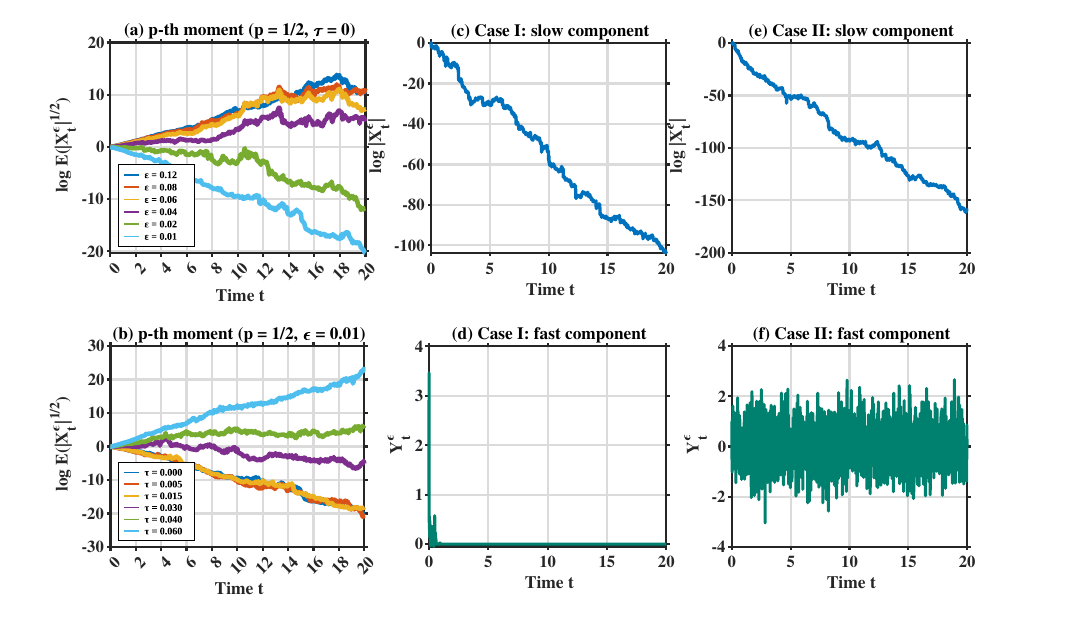}
    \caption{Numerical illustration of low-order moment stability
and the two mechanisms for exponential stability.
(a) Logarithmic empirical $p$th moments in Case I
for different values of $\eps$, with $\tau=0$.
(b) Logarithmic empirical $p$th moments in Case I
for different values of $\tau$, with $\eps=0.01$.
(c)--(d) Representative sample paths of $\log|X_t^\eps|$
and $Y_t^\eps$, respectively, in Case I.
(e)--(f) Representative sample paths of $\log|X_t^\eps|$
and $Y_t^\eps$, respectively, in Case II.
Panels~(c)--(f) correspond to $\eps=0.01$
and $\tau=0.005$.}
    \label{FIG-1}
\end{figure}


We consider the two systems \eqref{eq:I} and \eqref{eq:II}, both with initial data
$X_t^\eps=1$ for $-\tau\le t\le0$ and $Y_0^\eps=1$.

We first examine the effects of $\eps$ and $\tau$
on the low-order moment stability of Case I.
Using Euler--Maruyama with step size $h=10^{-3}$,
we estimate the logarithmic empirical $p$th moments with $p=1/2$
from $N=5000$ independent trajectories on $[0,20]$.
Panel~(a) of Figure~\ref{FIG-1} shows that decreasing $\eps$
changes moment growth into decay when $\tau=0$.
For fixed $\eps=0.01$, moment decay persists for
small delays but weakens as $\tau$ increases,
with growth observed for larger delays.
These observations are consistent with
Theorem~\ref{TH-3.16}.

To illustrate the two stability mechanisms,
we simulate both systems using the tamed
Euler--Maruyama method with $h=10^{-3}$,
$\eps=0.01$, and $\tau=0.005$.
Figure~\ref{FIG-1} also displays representative
sample paths on $[0,20]$.
In both cases, the slow component exhibits
exponential decay. In Case I, the fast component
also tends to zero
(Fig.~\ref{FIG-1}(c)--(d)), whereas in Case II
it continues to fluctuate
(Fig.~\ref{FIG-1}(e)--(f)).
This comparison illustrates the distinct
mechanisms underlying the exponential
stability results.

\subsection{High-Probability Stability under Local Lipschitz Conditions}
Consider the following system with initial data 
$Y_0^\eps=X_t^\eps\equiv 1,  t\in[-\tau,0]$.
\begin{small}\begin{equation}\label{eq-5.2}
\left\{
\begin{aligned}
{\rm d}X_t^\eps
&=\bigl[X_t^\eps-(X_t^\eps)^3\bigr]{\rm d}t
 +\bigl(2X_{t-\tau}^\eps+Y_t^\eps\bigr){\rm d}B_t^1,
 \qquad t\ge0,\\
{\rm d}Y_t^\eps
&=\frac1\eps\bigl(X_t^\eps-Y_t^\eps\bigr){\rm d}t
 +\frac1{2\sqrt\eps}Y_t^\eps\,{\rm d}B_t^2,
 \qquad t\ge0.
\end{aligned}
\right.
\end{equation}\end{small}
The coefficients are only locally Lipschitz, and Assumption~\ref{AS-4.1} and  Assumption~\ref{AS-4.2} hold with $L_M=3+3M^2$, $\vartheta_M=1$, $\lambda_{9,M}=3/2$, $c_{1,M}=1$ and $c_{2,M}=2$.
For each frozen $x$, the fast equation has invariant distribution
$\mu_x=\mathcal L(8x/G)$, where $G\sim\operatorname{Gamma}(9,1)$.
Its first two moments are $x$ and $8x^2/7$, which give
$\bar b(x)=x-x^3$ and $\bar a(x)=64x^2/7$.
For  $V(x)=|x|^{1/2}$,
$
\widetilde{\mathscr L}V(x)
=-\left(\frac9{14}+\frac{x^2}{2}\right)V(x)
\le-\frac9{14}V(x)(x\ne0).
$
This verifies Assumption~\ref{AS-4.3} and hence Theorem~\ref{TH-4.13} applies.

\begin{figure}[htbp]
    \centering
    \includegraphics[width=0.8\textwidth]{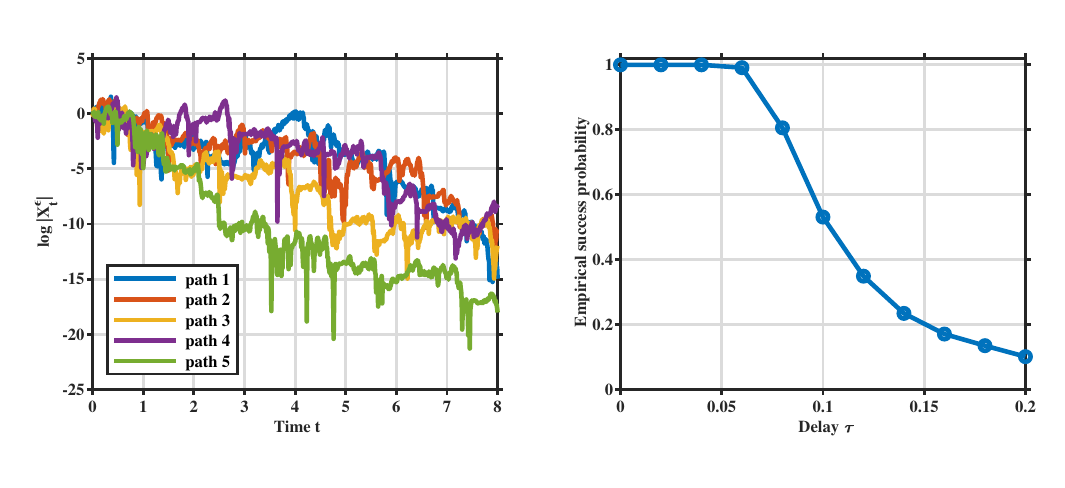}
    \caption{Numerical illustration of Eq.~\eqref{eq-5.2}.
    (a) Representative convergent sample paths of $X_t^\eps$ for
    $\eps=0.30$ and $\tau=0.02$. (b) Empirical stability success
    probabilities as a function of $\tau$ for $\eps=0.30$ and $T=200$,
    estimated from $N=5000$ independent trajectories for each
    value of $\tau$.}
    \label{FIG-3}
\end{figure}

With $h=10^{-3}$ and $\eps=0.30$, the sample paths are approximated by a tamed Euler--Maruyama scheme for the slow equation and Euler--Maruyama for the fast equation. Figure~\ref{FIG-3} illustrates the high-probability stability behavior of \eqref{eq-5.2}. The left panel shows five representative convergent sample paths for $\tau=0.02$ on a logarithmic scale. The right panel shows the empirical finite-horizon stability success probability, defined by $|X_T^\eps|<0.05$ at $T=200$ and estimated from $N=5000$ independent trajectories for each $\tau$. The success probability is high for small delays and decreases as $\tau$ increases, in agreement with Theorem~\ref{TH-4.13}.

\section{Concluding Remarks}\label{SEC-6}
We establish low-order practical exponential stability
for coupled slow--fast stochastic systems with small delays,
and further obtain moment and almost sure exponential
stability under two distinct mechanisms.
Under homogeneous fast dissipation, we extend the results
to locally Lipschitz systems, establishing exponential
stability with arbitrarily high probability.
Future work may address sharper delay and time-scale
bounds, weaker assumptions on the fast dynamics,
and high-probability stability under persistent
fast fluctuations and local Lipschitz conditions.



\bibliographystyle{siamplain-nolinks}
\bibliography{References}

\end{document}